%% file: submanifoldof-ellipsoid-Hari-2023_11m_04d.tex
\documentclass[12pt]{article}   
\usepackage[utf8]{inputenc}
\usepackage{amsmath,amssymb,amsthm,geometry, algorithm}
\usepackage{graphicx, verbatim, caption, subcaption}
\usepackage{xcolor}
\usepackage{fullpage}
\usepackage[colorlinks=true,linkcolor=blue]{hyperref}
\hypersetup{
   citecolor=blue,
   linkcolor=blue,
}
\usepackage[psamsfonts]{eucal}
\usepackage{microtype}
\date{\today}

\include{submanifold_macros}

\usepackage{booktabs} 
\usepackage{array} 
\usepackage{paralist} 
\usepackage{verbatim} 

\usepackage{fancyhdr} 
\usepackage{sectsty}
\allsectionsfont{\sffamily\mdseries\upshape} 

\usepackage[nottoc,notlof,notlot]{tocbibind} 
\usepackage[titles,subfigure]{tocloft} 

\usepackage{xcolor}
\newcommand{\ctext}{}

\newcommand{\mtext}{\color{black}} 

\newcommand{\ep}{\varepsilon}
\renewcommand{\G}{\mathcal G}
\newcommand\II{\mathcal S}
\newcommand\mS{\mathfrak{N}}

\newcommand\conv{\operatorname{conv}}
\newcommand\cone{\operatorname{cone}}
\newcommand\cl{\operatorname{cl}}

\title{Fitting a manifold to data in the presence of large noise}
\author{Charles Fefferman\thanks{Princeton University, Mathematics Department, Fine Hall, Washington Road, Princeton NJ, 08544-1000, USA.} \and  Sergei Ivanov\thanks{St. Petersburg Department of Steklov Institute of Mathematics, Russian Academy of Sciences, 27 Fontanka, 191023 St. Petersburg, Russia.} \and Matti Lassas\thanks{University of Helsinki, Department of Mathematics and Statistics, P.O. Box 68, 00014, Helsinki, Finland.} \and   Hariharan Narayanan\thanks{School of Technology and Computer Science, Tata Institute for Fundamental Research, Mumbai 400005, India.}}

\begin{document}
\maketitle
\begin{abstract}
We assume that $\MM_0$ is a $d$-dimensional $C^{2,1}$-smooth submanifold of $\R^n$.  Let $K_0$ be the convex hull of $\MM_0,$ and $B^n_1(0)$ be the unit ball.  We assume that  \beq \MM_0 \subseteq \partial K_0 \subseteq B^n_1(0). \lab{eq:new}\eeq  We also suppose that $\MM_0$ has volume ($d$-dimensional Hausdorff measure) less or equal to $V$, reach (i.e., normal injectivity radius) greater or equal to $\tau$.
 Moreover, we assume that $\MM_0$ is $R$-exposed, that is, tangent to every point $x \in \MM$ there is a closed ball in $\R^n$ of the radius $R$ that contains $\MM$.

Let $x_1, \dots, x_N$ be independent random variables sampled from uniform distribution on $\MM_0$ and
 $\zeta_1, \dots, \zeta_N$ be a sequence of i.i.d Gaussian random variables in $\R^n$ that are independent of $x_1, \dots, x_N$ and have mean zero and covariance $\sigma^2 I_n.$ 
We assume that we are given the noisy sample points $y_i$, given by 
$$
{\ctext y_i = x_i + \zeta_i,\quad \hbox{ for }i = 1, 2, \dots,N.} 
$$ 
Let $\eps,\eta>0$ be  real numbers and $k\geq 2$. {\mtext Given points $y_i$, $i=1,2,\dots,N$, we produce  a  $C^k$-smooth function which zero set is a manifold $\MM_{\mtext rec}\subseteq \R^n$ such that} the Hausdorff distance between $\MM_{\mtext rec}$ and $\MM_0$ is at most $ \eps$ and $\MM_{\mtext rec}$ has reach that is bounded below by ${c\tau}/{d{}^6}$ with probability at least $1 - \eta.$ Assuming $d < c \sqrt{\log \log n}$ and all the other parameters are positive constants independent of $n$, the number of the needed arithmetic operations is polynomial in $n$.
In the present work, at the cost of introducing a new condition that $\MM_0$ is $R$-exposed and requiring that $\MM_0$ be $C^{2, 1}$-smooth rather than $C^2,$ we allow the noise magnitude $\sigma$ to be an arbitrarily large constant, thus overcoming a drawback of the previous work \cite{filn0}. 
\end{abstract}
\tableofcontents

\section{Introduction}
One of the main challenges in high dimensional data analysis is dealing with the exponential growth of the computational and sample complexity of generic inference tasks as a function of dimension, a phenomenon termed ``the curse of dimensionality". One intuition that has been put forward to diminish the impact of this curse is that high dimensional data tend to lie near a low dimensional submanifold of the ambient space. Algorithms and analyses that are based on this hypothesis constitute the  subfield of learning theory known as manifold learning.
In the present work,  we give a solution to the following  question from manifold learning. 
Suppose data is drawn independently, identically distributed (i.i.d) from a measure supported on a low dimensional $C^{2,1}$ manifold $\MM_0$ whose reach is at least $\tau$, and corrupted by a {\it large} amount of (i.i.d) Gaussian noise.
How can can we produce a manifold $\MM_{\mtext rec}$ whose Hausdorff distance to $\MM_0$ is small and whose reach is not much smaller than $\tau$? 

This question is an instantiation of the problem of understanding the geometry of data. To give a specific real-world example, the issue of denoising noisy Cryo-electron microscopy (Cryo-EM) images falls into this general category.
Cryo-EM images are X-ray images of three-dimensional macromolecules, e.g. viruses, possessing an arbitrary orientation. The space of orientations is in correspondence with the Lie group $SO_3(\R)$, which is only three dimensional. However, the ambient space of greyscale images on $[0, 1]^2$ can be identified with an infinite dimensional subspace of $\mathcal{L}^2([0, 1]^2)$, which gets projected down to a finite $n-$dimensional subspace indexed by $n = k \times k$ pixels, where $k$ is large. through the process of dividing $[0, 1]^2$ into pixels. {When the molecule is not invariant under any nontrivial rigid body rotations,} the noisy Cryo-EM X-ray images lie approximately on an embedding of a compact $3-$dimensional manifold in a very high dimensional space. If the errors are modelled as being Gaussian, then fitting a manifold to the data can subsequently allow us to project the data onto this output manifold. Due to the large codimension and small dimension of the true  manifold,  the noise vectors are almost perpendicular to the true manifold and the projection would effectively denoise the data. The immediate rationale behind having a good lower bound on the reach is that this implies good generalization error bounds with respect to squared loss (See Theorem 1 in \cite{FMN}). Another reason why this is desirable is that the projection map onto such a manifold is Lipschitz within a tube of the manifold of radius equal to $c$ times the reach for any $c$ less than $1$. 

LiDAR (Light Detection and Ranging) also produces point cloud data for which the methods of this paper could be applied.

\subsection{A note on  constants}
In the following sections, we will denote positive absolute constants by $c, C, C_1, C_2, \oc_1$ etc. These constants are universal and positive, but their precise value may differ from occurrence to occurrence.
Also, for a natural number $n$, we will use $[n]$ to denote the set $\{1,2, \dots, n\}.$

\subsection{The model}\lab{subsec:pres-mod}
We assume that $\MM_0$ is a $d$-dimensional $C^{2,1}$-smooth submanifold of $\R^n$.  Let $K_0$ be the convex hull of $\MM_0$. We assume that  $$ \MM_0 \subseteq \partial K_0 \subseteq B^n_1(0). $$  We also suppose that $\MM_0$ has volume ($d$-dimensional Hausdorff measure) less or equal to $V$, reach (i.e. normal injectivity radius) greater or equal to $\tau$. 
We denote this class of manifolds by $\G(d, n, V, \tau).$
Let $x_1, \dots, x_N$ be a sequence of points chosen i.i.d at random from a probability measure $\mu$ that is proportional to the $d$-dimensional Hausdorff measure $\mathcal{H}^d_{\MM_0} = \la_{\MM_0}$ on $\MM_0$. 

 Let $G_\sigma^{(n)}$ denote the Gaussian distribution supported on $\R^n$ whose density (Radon-Nikodym derivative with respect to the Lebesgue measure) at $x$ is 
\beq \lab{eq:density} \pi_{G_\sigma^{(n)}}(x) = \left(\frac{1}{2 \pi \sigma^2}\right)^{\frac{n}{2}}  \exp \left(-\frac{\|x\|^2}{2 \sigma^2} \right).\eeq

Let $\zeta_1, \dots, \zeta_N$ be a sequence of i.i.d random variables independent of $x_1, \dots, x_N$ having the distribution $G_\sigma^{(n)}$. 
We observe 
\beq\label{y points}
{\ctext y_i = x_i + \zeta_i,\quad \hbox{ for }i = 1, 2, \dots,N.}
\eeq
Note that the distribution of $y_i$ (for each $i$), is the convolution of $\mu$ and $G_\sigma^{(n)}$.  This is denoted by $\mu*G_\sigma^{(n)}$. Let $\omega_d$ be the volume of a $d$ dimensional unit Euclidean ball.


We observe $y_1, y_2, \dots, y_N$ and in this paper  produce an $\eps$-net of $\MM_0$, where $\eps << \tau.$ In Section~\ref{sec:12},  we indicate how earlier results described below can be used to generate an implicit description of the manifold using this $\eps$-net.

\subsection{A survey of related work}\label{ssec:survey}

Let $f : K \rightarrow  \R$ be a function defined on a given (arbitrary) set $K \subset  \R^n$, and let $m \geq 1$ be
a given integer. The classical Whitney problem is the question  
 whether $f$ extends to a function $F \in C^m(\R^n)$ and  if such an
$F$ exists, what is the optimal $C^m$ norm of the extension. Furthermore, one is interested in
the questions
if the derivatives
of $F$, up to order $m$, at a given point can be estimated, or if 
one can construct extension $F$ so that it depends linearly on $f$.

{\color{black}
These questions go back to the work of H. Whitney \cite{W1,W11,W12} in 1934.
In the decades 
since
Whitney's seminal work, fundamental progress was made by G. Glaeser \cite{G}, Y. Brudnyi and
P. Shvartsman \cite{Br1,Br2,Br3,Br4,BS1,BS2} and \cite{Shv1,Shv2,Shv3}, and E. Bierstone-P. Milman-W. Pawluski \cite{BMP}. (See also N.
Zobin \cite{ZO1,ZO2} for the solution of a closely related problem.)

The above questions have been answered in the last few years, thanks to work of E. Bierstone,
Y. Brudnyi, C. Fefferman, P. Milman, W. Pawluski, P. Shvartsman and others, (see
\cite{BMP, Brom, Br1,Br3,Br4,BS2,F1,F2,F3,F5,F6}.) 
  Along the way,
the analogous problems with $C^m(\R^n)$ replaced by $C^{m,\omega }(\R^n)$, the space of functions whose
$m^{th}$ derivatives have a given modulus of continuity $\omega $, (see \cite{F5,F6}), were also solved.

The solution of Whitney's problems has led to a new algorithm for interpolation
of data, due to C. Fefferman and B. Klartag \cite{FK1,FK2}, where the authors show how to compute efficiently an interpolant $F(x),$ whose $C^m$ norm
lies within a factor $C$ of least possible,  where $C$ is a constant depending only on $m$ and $n.$ }

In traditional manifold learning, for instance, by using the ISOMAP algorithm introduced in the seminal paper  \cite{TSL}, one often aims to map points $X_j$ to points $Y_j=F(X_j)$ in an Euclidean space $\R^m$, where $m\geq n$ is as small as possible
so that the Euclidean distances $\|Y_j-Y_k\|_{\R^m}$ are close to the intrinsic  distances $d_M(X_j,X_k)$ and find a submanifold $\tilde M\subset \R^m$ that is close
to the points $Y_j$. { This method has turned out to be very useful, in particular in finding the topological manifold structure of the manifold $(M,g)$. 
It has been shown that when the original manifold  $(M,g)$  has a vanishing  Riemann curvature and satisfies certain 
convexity conditions, the manifold reconstructed by the ISOMAP approaches the original manifold  as the number of
the sample points tends to infinity (see the results in \cite{Bernstien,Donoho1,Donoho} for ISOMAP and \cite{Zha} for the continuum version of ISOMAP). We note that for a general Riemannian manifold, the construction  of}
a map $F:M\to \R^m$, for which { the intrinsic metric of} the embedded manifold $F(M)=\tilde M\subset \R^m$ is isometric to $(M,g)$  is  a very difficult 
task numerically as it means finding a map, the existence  of which is proved by the  Nash embedding theorem
(see \cite{Nash1,Nash2} and \cite{Verma1} on numerical techniques based on the Nash embedding theorem). 
{ We emphasize that the construction of an isometric embedding $f:M\to \R^n$ is outside of the context of the paper.}

%
%
One can  overcome the difficulties related {to the construction of the Nash embedding} by formulating the problem in a coordinate invariant way:
Given the geodesic distances of points sampled from a Riemannian manifold $(M,g)$, construct a manifold $M^*$ with an intrinsic metric tensor $g^*$ so that the Lipschitz distance of $(M^*,g^*)$  to the original manifold $(M,g)$ is small. 
{ The construction of abstract manifolds from the distances of sampled data points  has also been considered by Coifman and Lafon \cite{CoifmanLafon} and Coifman et al.\ \cite{diffusion2}
 using ``Diffusion Maps'',  and by Belkin and Niyogi \cite{BN} using ``EigenMaps'', where 
the data points are mapped to the values of the approximate eigenfunctions or diffusion kernels at the sample points. 
These methods construct 
a non-isometric embedding of the manifold $M$ into $\R^m$ with a sufficiently large $m$.
This construction  is continued in \cite{Meila} by computing an approximation the metric tensor $g$  by
using finite differences to find the Laplacian of  the products of the local coordinate functions. 
In \cite{FIKLN2}, we extend the results of \cite{FIKLN} that deals with the question how a smooth manifold, that approximates a manifold $(M,g)$,  can be constructed, when one is given the
 distances of the points of in a discrete subset $X$ of $M$ with small deterministic errors. In this paper we extend
 these results to two directions. First, the discrete set is randomly sampled and the distances have (possibly large) random errors. Second, we consider the case when some distance information is missing. }

The question of fitting a manifold to data is of interest to data analysts and statisticians \cite{aamari2019,  AizenbudSober, chen2015, Hein, Wasserman, Genovese:2012:MME:2188385.2343687,  kim2015, Sober, zhigang, STYau}.  We will focus our attention on results that provide an algorithm for describing a manifold to fit the data  together with upper bounds on the sample complexity. 

A  work in this direction  \cite{ridge}, building over \cite{Ozertem11} provides an upper bound on the Hausdorff distance between the output manifold and the true manifold equal to $O((\frac{\log N}{N})^{\frac{2}{n+8}}) + {O}(\sigma^2\log (\sigma^{-1}))$. Note that in order to obtain a Hausdorff distance of $c\eps$, one needs more than $\eps^{-n/2}$ samples, where $n$ is the ambient dimension. 
This bound is exponential in $n$ and thus differs significantly from our results.

\subsection{The case of small noise}
In an earlier work \cite{filn0},  we gave a solution to the following  question from manifold learning. 
Suppose data is drawn independently, identically distributed (i.i.d) from a measure supported on a low dimensional twice  differentiable ($\C^2$) manifold $\MM$ whose reach is $\geq \tau$, and corrupted by a small amount of (i.i.d) Gaussian noise.
How can can we produce a manifold $\MM_{\mtext rec}$ whose Hausdorff distance to $\MM$ is small and whose reach is not much smaller than $\tau$?


Let $\zeta_1, \dots, \zeta_N$ be a sequence of i.i.d random variables independent of $x_1, \dots, x_N$ having the distribution $G_\sigma^{(n)}$. 
We observe 
$$
{ y_i = x_i + \zeta_i,\quad \hbox{ for }i = 1, 2, \dots,N,}
$$
and wish to construct a manifold $\MM_{\mtext rec}$ close to  $\MM$  in Hausdorff  distance but at the same time having a reach not much less than $\tau$. Note that the distribution of $y_i$ (for each $i$), is the convolution of $\mu$ and $G_\sigma^{(n)}$.  This is denoted by $\mu*G_\sigma^{(n)}$. Let $\omega_d$ be the volume of a $d$ dimensional unit Euclidean ball.
{ In \cite{filn0}, we supposed that 
\beq\label{eq:sigma-May}
   \sigma< r_{c} D^{-1/2},\quad \hbox{where }
 r_{c} := cd^{-C}\tau, \  D = \min\left(n, \frac{V}{c^d\omega_d\beta^d}\right),\ \beta =  \tau\sqrt{\frac{c^d \omega_d \tau^d}{V}},\hspace{-15mm}
\eeq 
and $\Delta \geq \frac{Cd\sigma^2}{\tau}.$
The points  $y_1, y_2, \dots, y_N$ are observed and for $k \geq 3$,  the algorithm produces a description of a $\C^k-$manifold $\MM_{\mtext rec}$ such that the Hausdorff distance between $\MM_{\mtext rec}$ and $\MM$ is at most $ \Delta$ and $\MM_{\mtext rec}$ has reach that is bounded below by $\frac{c\tau}{d^6}$ with probability at least $1 - \eta.$ 

\subsection{New contributions}

In the present work, we allow $\sigma$ to be an arbitrarily large constant, thus overcoming a drawback of the previous work from \cite{filn0}  mentioned above.
On the other hand, our model is more restrictive in some ways; namely, we consider a $C^{2, 1}$ manifold $\MM$  rather than a $\C^2$ manifold, and require that the manifold underlying the data is $R$-exposed in the sense of Section~\ref{sec:R}. This means that, tangent to every point $x \in \MM$ is a $n-1$-sphere of radius $R$, such that, the closed ball that it is the boundary of, contains $\MM$.

The following is our main theorem.



\begin{theorem}\lab{thm:main}
Let the dimensions $d,n$, the noise level $\sigma$ and the 
geometric bounds $R, \Lambda, \tau, V$, see formulas (G1), (G2), and (G3) in
Section \ref{subsec Geometric bounds}, be such that 
that   
 $\Lambda\ge\tau^{-2}$, see \eqref{assumption 1}.
Let the probability bound $\eta$ satisfy $0<\eta<\frac 12$.
Moreover, let the accuracy parameter $\eps$ be such that 
 $\eps< \beta^2/2$, see \eqref{formula eps},
where $\beta$ is given by \eqref{def: beta} and \eqref{def: alpha}.
Assume that we are given noisy sample points $y_1,\dots,y_N \in \R^n$, see \eqref{y points}, where 
$N$
satisfies 
\beq\label{bound for Ncp}
 N\ge \tilde{\Omega}\left(\exp\left(\left(\frac{\sigma}{\eps\tau}\log \frac{V}{\tau^d}\right)^2\right) \log \left( \eta^{-1}\right)\right).  
 \eeq
Then,
 our algorithm in Sections~\ref{sec:correct} and \ref{sec:12} produces the parameters of a function 
 $F_{\mtext rec}$, see (\ref{eq:master}), such the manifold $\MM_{\mtext rec}=\{x\in \R^n:\ F_{\mtext rec}(x)=0\}$ has a reach at least $\frac{C\tau}{d^6}$ 
and the Hausdorff distance of $\MM_{\mtext rec}$ to $\MM_0$ is less than $\eps,$ with probability at least $1 - \eta.$ 
{\mtext  Moreover, when $d < c\sqrt{\log\log n}$, the number of arithmetic operations in the algorithm is less than $n^{C_0}$, where $C_0$ depends only on 
$ \sigma, R, \Lambda, \tau, V, \eta, \eps$.}
\end{theorem}

\section{Preliminaries}

\subsection{Notation for manifolds}\lab{subsec:2.1}

\begin{itemize}
\item 
$\MM$ is a closed $d$-submanifold of $E$, $p\in \MM$, where $E$ is an $n$-dimensional Euclidean space which we identify with $\R^n.$

\item
$T_p\MM$ denotes the tangent space to $\MM$ at $p$,
regarded a linear subspace of $\R^n$.

\item
$T_p^\perp\MM=(T_p\MM)^\perp$ is the normal space, i.e., the orthogonal complement
to $T_p\MM$ in $\R^n$.

\item
$\II_p=\II_p^\MM$ is the second fundamental form of $M$ at $p$.
It is a symmetric bilinear map from $T_p\MM\times T_p\MM$ to $T_p^\perp\MM$.
To simplify the technical details in the sequel, we assume that $\II_p$
is extended to a symmetric bilinear map from $\R^n\times\R^n$
to $\R^n$ by setting $\II_p(v,w)=0$ whenever $v\in T_p^\perp\MM$.
The values of the extended $\II_p$ still belong to $T_p^\perp\MM\subset\R^n$.

\item
For a linear subspace $X\subset\R^n$, we denote by $\Pi_X$ the
orthogonal projection from $\R^n$ to~$X$.
When the manifold $\MM$ is clear from context,
we use notation $\Pi_p$ and $\Pi_p^\perp$, where $p\in M$,
for the projections $\Pi_{T_p\MM}$ and $\Pi_{T_p^\perp\MM}$, respectively.
\end{itemize}

\subsection{Notation for convex sets and cones}

	\begin{figure}\centering
		\includegraphics[width=0.4\linewidth]{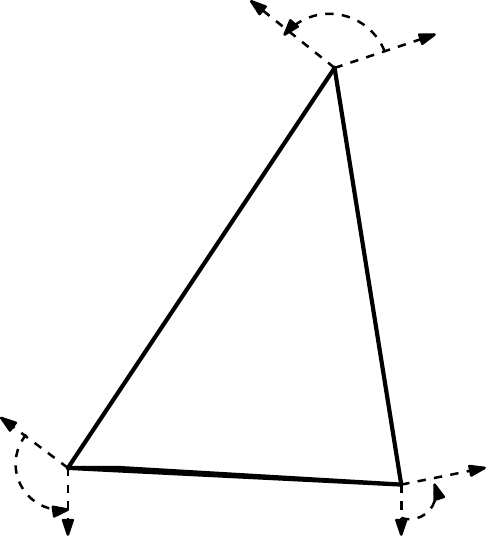}	
		\caption{Outer normal cones for points on $\MM \subseteq \partial K$. When $\MM$ is a zero dimensional manifold corresponding to three noncollinear points, $K$  is a triangle and the outer normal cones are cones at the vertices of the triangle as depicted.}\end{figure}

\begin{itemize}
\item 
$B_r(x)$, where $x\in\R^n$ and $r>0$, is the $r$-ball centered at $x$.

\item
$\conv(X)$, for $X\subset\R^n$, is the convex hull of $X$.

\item
A set $K\subset\R^n$ is a \textit{cone} if $tx\in K$
whenever $x\in K$ and $t\ge 0$.
All cones are linear cones, i.e., with apex at 0.

\item
$\cone(X)$ is the least cone containing $X$:
$\cone X = \{tx\mid x\in X, t\ge 0\}$.
One sees that the least convex cone containing $X$ can be obtained as
$$
 \conv\cone(X)=\cone\conv(X)
 = \{ \text{nonnegative linear combinations of points of $X$} \}.
$$

\item
For a set $K\subset\R^n$, $K^\circ$ denotes the polar set:
$$
 K^\circ = \{ x\in\R^n : \forall y\in K,\ \langle x,y\rangle \le 1 \}
$$
If $K$ is a closed convex set and $0\in K$, then $K^{\circ\circ}=K$
by the Bipolar Theorem.
If $K$ is a cone then the definition of $K^\circ$ can be simplified:
$$
 K^\circ = \{ x\in\R^n : \forall y\in K,\ \langle x,y\rangle \le 0 \}.
$$

\item For a closed convex set $K\subset\R^n$ and $p\in K$, the tangent cone
of $K$ at $p$ is
$$
 T_K(p) = \cl\cone(K-p)
$$
and the outward normal cone is
$$
 N_K(p) = (T_K(p))^\circ 
 = \{ v\in\R^n : \forall x\in K,\ \langle v, x-p\rangle \le 0 \} .
$$

\item The distance between cones $K_1$ and $K_2$ is defined
as the Hausdorff distance between their intersections with the unit ball:
$$
 d_{CH}(K_1,K_2) = d_H(K_1\cap B_1(0), K_2\cap B_1(0))
$$
where $d_H$ is the Hausdorff distance.
\end{itemize}

\subsection{Geometric bounds}\label{subsec Geometric bounds}

We assume the following about the manifold $\MM=\MM^d\subset\R^n$.

\begin{enumerate}
\item [{\mtext (G1)}]
The reach of $\MM$ is bounded below by a constant $\tau>0$, and is $C^{2, 1}$ thus belongs to $\G(d, n, V, \tau)$.

\item [{\mtext (G2)}]
$\MM$ is \textit{$R$-exposed} for some constant $R>0$.
The $R$-exposedness property is defined as follows:
We say that a point $p\in \MM$ is $R$-exposed in $\MM$
if there exists a closed ball $B\subset\R^n$ of radius $R$
such that $\MM\subset B$ and $p$~belongs to the boundary of~$B$.
The manifold $\MM$ is called $R$-exposed
if all its points are $R$-exposed. The condition of a manifold being $R-$exposed for some finite $R$  is an open condition with respect to the $C^{1, 1}$- topology. 
\item [{\mtext (G3)}]
The second fundamental form of $\MM$ is Lipschitz: There exists $\Lambda>0$
such that
$$
 \|\II^\MM_x-\II^\MM_y\| \le \Lambda \|x-y\|
$$
for all $x,y\in \MM$, where $\II^\MM_x$ is the second fundamental form of $\MM$
at $x$ extended to $\R^n\times\R^n$ as explained above,
and $\|\cdot\|$ in the left-hand side is the operator norm.
\end{enumerate}

{\bf Notation:} All large absolute constants below are denoted by the same letter $C$ and small absolute constants are denoted $c$.

\subsection{Federer's reach criterion}

Recall that the \textit{reach} of a closed set $A\subset\R^n$
is the supremum of all $r\ge 0$ such that for every point
$x\in\R^n$ with $dist(x,A)\le r$ there exists a unique
nearest point in~$A$.
The following result of Federer \cite[Theorem 4.18]{federerpaper},
gives an alternate characterization of the reach of a submanifold of a Euclidean space.

\begin{proposition}[Federer's reach criterion]\label{thm:federer}
Let $A\subset\R^n$ be a closed set. Then
$$
\reach(A)^{-1} = 
\sup\left\{2|q-p|^{-2}dist(q-p, Tan(A,p)) \mid p, q \in A, p \neq q\right\}
$$
where $Tan(A,p)$ is the set of all tangent vectors of $A$ at~$p$,
see \cite[Definition 4.3]{federerpaper}.
\end{proposition}

\begin{corollary}\label{cor:reach condition}
A manifold $\MM\subset\R^n$ has $\reach(\MM)\ge\tau$
if and only if
\beq\label{e:reach condition}
|\Pi_{T_p^\perp\MM}(q-p)| \le \frac1{2\tau}|q-p|^2
\eeq
for all $p,q\in\MM$.
\end{corollary}

\begin{proof}
If $A=\MM$ then $Tan(A,p)=T_p\MM$
and $dist(q-p, T_p\MM)=|\Pi_{T_p^\perp\MM}(q-p)|$.
Thus the corollary is a reformulation of Federer's reach criterion.
\end{proof}

\begin{corollary}\label{cor:reach2}
Let $\MM\subset\R^n$ be a manifold with $\reach(\MM)\ge\tau$.
Then for all $p,q\in\MM$ such that $|q-p|<\tau$,
$$
 |\Pi_{T_p^\perp\MM}(q-p)|
 \le \frac{|\Pi_{T_p\MM}(q-p)|^2}{\tau}
$$
\end{corollary}

\begin{proof}
Let $x=|\Pi_{T_p\MM}(q-p)|$ and $y=|\Pi_{T_p^\perp\MM}(q-p)|$.
Then \eqref{e:reach condition} takes the form
$$
y\le \frac1{2\tau}(x^2+y^2) .
$$
Solving this as a quadratic inequality in $y$, where $0\le y<\tau$,
we obtain that
$$
 y \le \tau - \sqrt{\tau^2-x^2} \le x^2/\tau
$$
where the second inequality comes from the trivial estimate
$\sqrt{1-x^2/\tau^2} \ge 1-x^2/\tau^2$.
\end{proof}

\section{$R$-exposedness}\lab{sec:R}

	\begin{figure}\centering
		\includegraphics[width=0.5\linewidth]{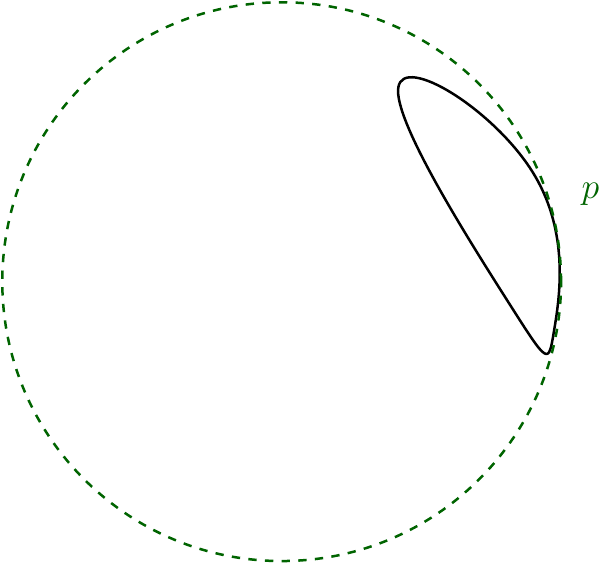}
		\caption{The manifold $\MM_0$ is $R$-exposed. It is tangent to the dashed sphere at a point $p$ and the corresponding ball contains $\MM_0.$}
	\end{figure}

\begin{lemma}\label{l:center condition}
Let $A\subset\R^n$ be a closed set, $p\in A$ and $z\in\R^n$.
Then $p$ is a furthest point from~$z$ in~$A$ if and only if
\beq\label{e:center condition}
\langle z-p, q-p\rangle \ge \frac12 |q-p|^2
\eeq
for all $q\in A$.
\end{lemma}

\begin{proof}
From the identity
$$
 |z-q|^2 = |(z-p)-(q-p)|^2 = |z-p|^2 + |q-p|^2 - 2\langle z-p, q-p\rangle
$$
one sees that \eqref{e:center condition} is equivalent to the
inequality $|z-q|^2 \le |z-p|^2$, and the lemma follows.
\end{proof}

\begin{lemma} \label{l:nu_p}
A manifold $\MM\subset\R^n$ is $R$-exposed if and only if
for every $p\in \MM$ there exists a unit vector $\nu_p\in T_p^\perp\MM$
such that
\beq\label{e:nu_p}
 \langle q-p, \nu_p \rangle \ge \frac1{2R} |q-p|^2
\eeq
for all $q\in M$.
\end{lemma}
 
\begin{proof}
A point $p\in\MM$ is $R$-exposed if and only if there exists
$z\in\R^n$ (the center of the $R$-ball from the definition) such that
$|z-p|=R$ and $p$ is furthest from $z$ in~$\MM$.
For $z=p-R\nu_p$, where $\nu_p$ is a unit vector,
the inequality \eqref{e:center condition} is equivalent to \eqref{e:nu_p}.
Differentiating \eqref{e:nu_p} with respect to~$q$ at $q=p$ shows that
$\nu_p\in T_p^\perp\MM$.
\end{proof}

\subsection{$\eps$-denseness of the $R$-exposedness condition}

	\begin{figure}\centering
		\includegraphics[width=0.5\linewidth]{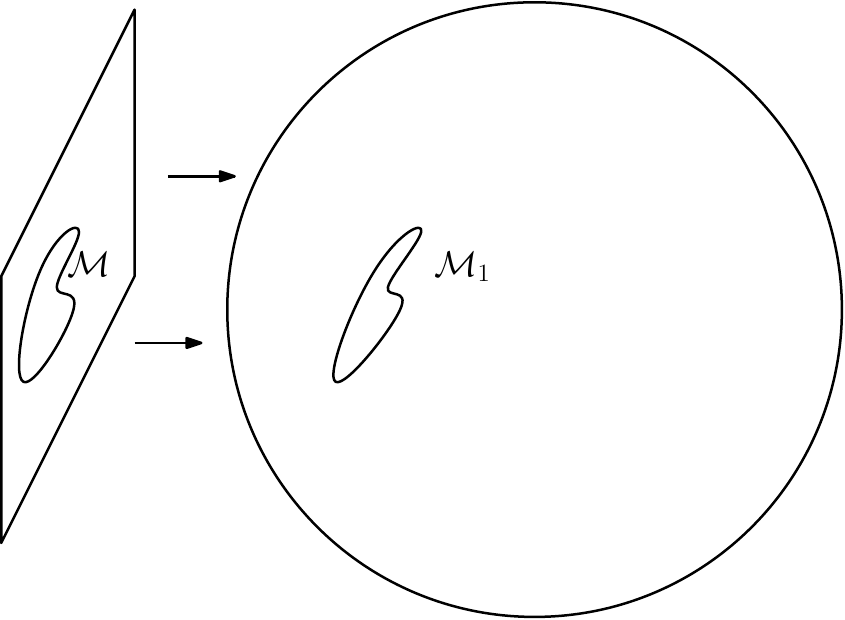}	
		\caption{Project a flattened version $\MM:= \Pi_H \MM_0$ of $\MM_0$ onto a sphere of large radius. The image $\MM_1$ of this projection satisfies the $R$-exposedness condition}\end{figure}

The condition of $R$-exposedness is $\eps$-dense in the sense of the following lemma.

\begin{lemma}
Let $0 < \eps < c.$ 
Let $\MM$ belong to $\G(d, n, V, \tau).$ Then, {\mtext if  $n > \frac{CV}{\omega_d}(\min(\frac{\tau}{2d}, \frac{\eps\tau}{2}))^{-d}$, there} exists a manifold $\MM_1 \in \G(d, n, V, \tau(1 - C\eps))$ such that the Hausdorff distance $d_H(\MM, \MM_1) < \eps\tau$ and $\MM_1$ is $R$-exposed for  $R = (\eps\tau)^{-1}.$
\end{lemma}
\begin{proof}
Let $\NN$ be a finite subset of $\MM$ of minimum size, such that for all $x \in \MM$ there exists $z \in \NN$ with $|x - z| < \eps\tau.$ 
\begin{claim}\lab{cl:3.1} $|\NN| \leq n.$
\end{claim}
\begin{proof}[Proof of Claim~\ref{cl:3.1}.] 
The cardinality of a minimum $\eps'$-cover equals is less or equal to the cardinality of a maximal set of disjoint $\eps'/2$-balls (with respect to the Euclidean metric).
Let $p \in \MM$ and $\MM_p := \MM \cap B_{\min(\tau/2d, \eps\tau/2)}^D(p)$. Then, by Lemma~\ref{lem:6} for every $q \in \MM_p$,
the projection $\Pi_q$ on to the tangent space at $q$ and the projection $\Pi_p$ on to the tangent space at $p$ satisfy 
$\|\Pi_q - \Pi_p\| < \min(C\eps, Cd^{-1}).$ It follows that the volume of $\MM_p$ is greater than $c \omega_d \min(\frac{\tau}{2d}, \frac{\eps\tau}{2})^d.$ Let $\NN'$ be a maximal disjoint set of $\min(\frac{\tau}{2d}, \frac{\eps\tau}{2})-$balls in $\MM$ (with respect to the Euclidean metric). Due to the disjointness of the different $\MM_p$ when $p$ ranges over $\NN'$, $$\sum_{p\in \NN'} \vol \MM_p \leq \vol \MM.$$ Therefore, $$|\NN| \leq |\NN'| \leq  \frac{CV}{\omega_d \min(\frac{\tau}{2d}, \frac{\eps\tau}{2})^d}.$$ The claim follows.

\end{proof}
Let $H$ be the unique codimension one hyperplane containing $\NN.$ Recall that $\MM \subseteq  B_1^n(0).$ Let $\tilde{B}_R$ denote the unique ball radius $ R = (\eps\tau)^{-1},$ that is tangent to  $H,$ at the foot of the perpendicular of the origin to $H.$ Every point in $\Pi_H\MM$ is at a distance of at most $C\eps\tau$ from $H \cap B_1^n(0)$. It is also true that every point in $H\cap B_1^n(0)$ is at a distance of at most $C\eps\tau$ from $  B^n_1(0)\cap \partial \tilde{B}_R.$ Therefore every point in $\MM$ is at a distance of at most $C\eps\tau$ from  $B^n_1(0)\cap \partial \tilde{B}_R, $ and hence from $\partial \tilde{B}_R.$ Let $\Pi_R := \Pi_{\tilde{B}_R}$ denote the projection map from $Tub_{\frac{R}{2}}\left(\partial \tilde{B}_R\right),$ the tubular neighborhood of $\partial \tilde{B}_R$ of radius $\frac{R}{2},$ to $\partial \tilde{B}_R$ that maps each point to the nearest point on $\partial \tilde{B}_R$.  We proceed to prove the following claim.
\begin{claim}\lab{cl:3.2}  $\MM_1 := \Pi_R(\Pi_H \MM)$ belongs to $\G(d, n, V, \tau(1 - C\eps))$. \end{claim}
\begin{proof}[Proof of Claim~\ref{cl:3.2}]
It follows from Lemma 3.3, page 23 of \cite{filn0}, that \beq \Pi_H \MM \in   \G(d, n, V, \tau(1 - C\eps)).\eeq Therefore, in order to prove this claim, it suffices to show  that the reach of $\Pi_H\MM$ does not decrease by more than $C\eps\tau$ under the application of $\Pi_R.$
Let $p$ and $q$ be two points on $\Pi_H \MM.$ Let $o$ be the center of $\tilde{B}_R.$  Let the $d+1$-dimensional subspace containing  $Tan(p, \Pi_H \MM)$ and $o$ be $S.$  The sine of the angle between $qo$ and  the projection of $qo$ on $S$ is less than $2R^{-1} = 2\eps\tau,$ and the cosine of the angle between $qo$ and any vector in $Tan(p, \Pi_H \MM)$ is less than $C\eps\tau.$ Similarly, the sine of the angle between $qo$ and  a normal to $H$ is less than $2R^{-1} = 2\eps\tau.$ Thus the operator norm $\|\Pi_1 - \Pi_2\|$ where $\Pi_1$ is the orthoprojection on to $Tan(p, \Pi_H \MM)$ and $\Pi_2$ is the orthoprojection on to $Tan(\Pi_R p, \Pi_R\Pi_H \MM),$ is less than $C\eps\tau.$ Therefore, 
$$dist(\Pi_R q, Tan(\Pi_R p, \Pi_R\Pi_H \MM)) \leq (1 + C\frac{\eps\tau}{R})dist(q, Tan(p, \Pi_H \MM)).$$ 
$\Pi_R$ restricted to $H$ is a real analytic diffeomorphism from $H$ to an open hemisphere in $\partial\tilde{B}_R$ and a contraction with respect to the usual Euclidean and Riemannian metrics respectively on $H,$ and $\partial\tilde{B}_R.$ Together with Proposition~\ref{thm:federer} the claim follows.
\end{proof}
Note that $\MM_1$  is $R$-exposed because $\MM_1 \subseteq \partial \tilde{B}_R.$
\end{proof}

\subsection{Derivative estimates}

\begin{lemma}\label{l:local graph}
Let $\MM\subset\R^n$ be a closed $C^{2, 1}-$submanifold with $\reach(\MM)\ge\tau$.
Fix $p\in\MM$ and define open sets $U\subset T_p\MM$ and $Q\subset\R^n$ by
$$
  U = \{x\in T_p\MM : |x|<\tfrac{\tau}{8} \}
$$
and
$$
 Q = \{ x\in\R^n : |\Pi_{T_p\MM}(x)| < \tfrac{\tau}{8}
 \text{ and } |\Pi_{T^\perp_p\MM}(x)| < \tfrac{\tau}{8} \} .
$$
Then the set $(\MM-p) \cap Q$
is a graph of a $C^{2,1}$ function
$$
 f \colon U \to T_p^\perp\MM
$$
such that 
the following estimates hold for 
all $x\in U$ and some absolute constant $C>0$:
\begin{align}
\label{e:f bound}
 |f(x)| 
 &\le \tau^{-1} |x|^2 , \\
\label{e:df bound}
 \|d_xf\| &\le C\tau^{-1}|x| , \\
\label{e:d2f bound}
 \|d^2_xf\| &\le C \tau^{-1}
\end{align}
where $d_xf$ and $d^2_xf$ are the first and second differentials
of $f$ at $x$, and $\|d_xf\|$ and $\|d^2_xf\|$ are their operator norms.
 
If, in addition, the second fundamental form of $\MM$ is $\Lambda$-Lipschitz
with $\Lambda\ge\tau^{-2}$,
then the map $x\mapsto d_x^2f$ is $C\Lambda$-Lipschitz:
\beq\label{e:f2f lip bound}
 \|d^2_xf-d^2_yf\| \le C\Lambda |x-y|
\eeq
for all $x,y\in U$.
\end{lemma}

\begin{proof}
The existence of $f$ and the estimates \eqref{e:f bound}, \eqref{e:df bound},
\eqref{e:d2f bound} follow from \cite[Lemma A.2]{filn0}.
It remains to prove~\eqref{e:f2f lip bound}.
Throughout the proof we use the short notation
$\Pi_q=\Pi_{T_q\MM}$ and $\Pi_q^\perp=\Pi_{T_q^\perp\MM}$
for the orthogonal projections to the tangent and normal spaces at $q\in\MM$.
All absolute constants are denoted by the same letter~$C$.

Consider the local parametrization $\varphi$ of $\MM$ determined by $f$,
that is $\varphi\colon U\to\R^n$ is a map given by
$$
 \varphi(x) = x + f(x), \qquad x\in U .
$$
For every $x\in U$, the parametrization and the second fundamental form of
$\MM$ at $\varphi(x)$ are related by the formula
\beq \label{e:II via varphi}
 \II_{\varphi(x)} ( d_x\varphi(v), d_x\varphi(w)) )
 = \Pi_{\varphi(x)}^\perp(d_x^2\varphi(v,w))
 = \Pi_{\varphi(x)}^\perp(d_x^2 f(v,w))
\eeq
for all $v,w\in T_p\MM$.
This formula defines $\II_{\varphi(x)}$ on the tangent space.
By our convention, $\II_{\varphi(x)}$ is extended
to the whole $\R^n$ via projection:
$$
 \II_{\varphi(x)}(\xi,\eta)
 = \II_{\varphi(x)}(\Pi_{\varphi(x)}(\xi), \Pi_{\varphi(x)}(\eta))
$$
for all $\xi,\eta\in\R^n$.
These identities imply that
$
 \|\II_{\varphi(x)}\| \le C\tau^{-1}
$
for all $x\in U$.

Fix $x,y\in U$ sufficiently close to each other
and a unit vector $v\in T_pM$.
Let $\delta=|x-y|$, $v_x=d_x\varphi(v)$ and $v_y=d_y\varphi(v)$.
By \eqref{e:df bound} and \eqref{e:d2f bound} we have
$|v_x| \le C$, $|v_y|\le C$,
\beq\label{e:df Lip bound}
 \|d_xf-d_yf\| \le C\tau^{-1}\delta
\eeq
and hence $|v_x-v_y| \le C\tau^{-1}\delta$.
By the $\Lambda$-Lipschitz continuity
of the second fundamental form,
$$
 |\II_{\varphi(x)}(v_x,v_x) - \II_{\varphi(y)}(v_x,v_x)|
 \le \Lambda \cdot |\varphi(x)-\varphi(y)| \cdot |v_x|^2
 \le C\Lambda\delta .
$$
From the above bounds on $\|\II\|$, $|v_x|$, $|v_y|$, $|v_x-v_y|$,
one sees that
$$
 |\II_{\varphi(y)}(v_x,v_x)-\II_{\varphi(y)}(v_y,v_y)|
 \le C\|\II_{\varphi(y)}\|\cdot|v_x|\cdot|v_x-v_y|
 \le C\tau^{-2}\delta \le C\Lambda\delta
$$
where the last inequality follows from the assumption $\Lambda\ge\tau^{-2}$.
Thus
$$
 |\II_{\varphi(x)}(v_x,v_x) - \II_{\varphi(y)}(v_y,v_y)| \le C\Lambda\delta .
$$
By \eqref{e:II via varphi} this is equivalent to
\beq\label{e:graph II diff}
 |\Pi_{\varphi(x)}^\perp(d_x^2f(v,v)) - \Pi_{\varphi(y)}^\perp(d_y^2f(v,v)) | 
 \le C\Lambda\delta .
\eeq
The estimate \eqref{e:df Lip bound} implies a bound for
the distance between the tangent planes
$T_{\varphi(x)}\MM$ and $T_{\varphi(y)}\MM$ (see \eqref{e:distance between planes mod2}):
$$
 \| \Pi_{\varphi(y)}^\perp - \Pi_{\varphi(x)}^\perp \|
 =  \| \Pi_{\varphi(y)} - \Pi_{\varphi(x)} \| \le C\tau^{-1}\delta ,
$$
therefore
$$
 | \Pi_{\varphi(y)}^\perp(d_y^2f(v,v)) - \Pi_{\varphi(x)}^\perp(d_y^2f(v,v)) |
 \le C\tau^{-1}\delta \|d_y^2f\| \le C\tau^{-2}\delta \le C\Lambda\delta .
$$
By \eqref{e:graph II diff} it follows that
\beq\label{e:proj d2f diff}
 \left|\Pi_{\varphi(x)}^\perp\big(d_x^2f(v,v) - d_y^2f(v,v)\big)\right|
 \le C\Lambda\delta .
\eeq
Let $P\colon T_p^\perp\MM \to T_{\varphi(x)}^\perp\MM$
be the restriction of $\Pi_{\varphi(x)}^\perp$ to $T_p^\perp\MM$.
Since $\Pi_{\varphi(x)}$ is the graph of the linear map
$d_xf\colon T_p\MM\to T_p^\perp\MM$ satisfying 
$
\|d_xf\| \le C\tau^{-1}|x| \le C
$
(see \eqref{e:df bound}), $P$ is bijective and $\|P^{-1}\|\le C$.
This and \eqref{e:proj d2f diff} imply that
$$
 |d_x^2f(v,v) - d_y^2f(v,v)|
 = \left|P^{-1}\circ\Pi_{\varphi(x)}^\perp\big(d_x^2f(v,v) - d_y^2f(v,v)\big)\right|
 \le C\Lambda\delta
$$
and \eqref{e:d2f bound} follows.
\end{proof}

\begin{lemma} \label{l:C3 bound implication}
Let $f$ be as in Lemma \ref{l:local graph}.
Then for all  $x,y\in T_p\MM$ such that $\max\{|x|,|y|\}\le\frac14\tau$,
$$
 | f(x+y)-f(y) - d_yf(x) -f(x) | \le C\Lambda |x|^2|y| .
$$
\end{lemma}

\begin{proof}
Define $h\colon[0,1]\to T_p^\perp\MM$ by
$$
 h(t) = f(tx+y)-f(y)- t d_yf(x) - f(tx) .
$$
We have $h(0)=0$, $h'(0)=0$ and 
$$
|h''(t)| = | d^2_{tx+y}f(x,x) - d^2_{tx}f(x,x) | \le C\Lambda|x|^2|y|
$$
from the Lipschitz condition on $d^2f$.
Hence $|h(1)| \le C\Lambda|x|^2|y|$ as claimed.
\end{proof}

\section{Geometric bounds preserved under projection}

\subsection{Preservation of uniform exposedness by projection}

Here we need the $\Lambda$-Lipschitz continuity of the second
fundamental form.
{\bf We   assume below that}
\beq\label{assumption 1} \Lambda\ge\tau^{-2} \eeq where $\tau$ is the reach bound,
otherwise $\Lambda$ should be replaced by $\max\{\Lambda,\tau^{-2}\}$
in some formulas.

We will need the following well-known facts
about distances between linear subspaces:
When $X,Y\subset\R^n$ are linear subspaces with $\dim X=\dim Y$,
we define analogously to  \cite{Kato},  Chapter IV, section 2.1,
\begin{equation}\label{eq: def distance}
\begin{aligned}
\delta(X,Y)
  &=\sup \{ \dist(x,Y) : x\in X\cap B_1^n \}\\
    &=\sup \{ \dist(x,Y\cap B_1^n) : x\in X\cap B_1^n \},
\end{aligned}
\end{equation} and $B_1^n$ is the unit ball in $\R^n$ centered at~0.
Note that it always holds that $0\le \delta(X,Y)\leq 1$.

If $\delta(X,Y)<1$ then
\cite[Lemma 221]{Kato2} implies that either $\delta(Y,X)= \delta(X,Y)$
or $Y$ has a proper subspace $Y_0$ that is isomorphic to $X$.
As   $\dim X=\dim Y$, the latter is not possible and hence $\delta(X,Y)<1$ implies that $\delta(Y,X)= \delta(X,Y)<1$.
By changing roles of $X$ and $Y$ we see that $\delta(Y,X)<1$ implies that $\delta(X,Y)= \delta(Y,X)<1$.
Thus we see that either both $\delta(X,Y)$ and $\delta(Y,X)$ are equal to $1$,  or both $\delta(X,Y)$ and $\delta(Y,X)$ 
are strictly less than 1 and $\delta(X,Y)= \delta(Y,X)$. These arguments show that in all possible  cases
\begin{equation}\label{eq: symmetry of delta}
\begin{aligned}
  \delta(X,Y)=\delta(Y,X).
\end{aligned}
\end{equation}
We observe that
$$
\hat \delta(X,Y):=\max(\delta(X,Y),\delta(X,Y))= d_H(X\cap B_1^n,Y\cap \B_1^n).
$$
In the case when $\delta(X,Y)= \delta(Y,X)=1$, there is there is non-zero vector $x\in X\cap Y^\perp$ and we see that $\| \Pi_X-\Pi_Y\|=\hat \delta(X,Y)=1$.
On the other hand, when  $\hat \delta(X,Y)<1$,   \cite[Theorem I-6.34]{Kato},  and \cite[Lemma 221]{Kato} imply
that 
\begin{equation}\label{eq: difference of projections}
\| \Pi_X-\Pi_Y \|=\|(I- \Pi_X)\Pi_Y \|=\|(I- \Pi_Y)\Pi_X \|=\hat \delta(X,Y).
\end{equation}
In particular, this implies that
 \begin{equation}\label{e:distance between planes mod2}
 \hat \delta(X,Y)=\|\Pi_X-\Pi_Y\|= \|\Pi_{X^\perp}-\Pi_{Y^\perp}\|=\hat \delta(X^\perp,Y^\perp).
\end{equation}


\begin{lemma}\label{l:proj-angle}
Let $\MM\subset\R^n$ be a manifold with $\reach(\MM)\ge\tau$
and $S\subset\R^n$ a linear subspace such that
$$
\sup_{x \in \MM} dist(x, S) \le \alpha^2\tau
$$
for some $\alpha\in(0,\frac14)$.
Then for every $p\in\MM$ and every unit vector $v\in T_p\MM$,
$$
 dist(v,S) = |\Pi_{S^\perp}(v)| \le 3\alpha .
$$
\end{lemma}

\begin{proof}
Let $p\in\MM$ and let $v\in T_p\MM$ be unit vector.
We borrow from \cite{filn0} the following fact
(see \cite[Lemma A.1]{filn0}): the set
$ \Pi_{T_p\MM}(\MM\cap B_\tau(p)) $
contains the ball of radius $\frac{\tau}{4}$ in $T_p\MM$ centered at
$\Pi_{T_p\MM}(p)$.
Hence there exists $q\in\MM$ such that $|q-p|<\tau$
and $\Pi_{T_p\MM}(q-p) = \alpha\tau v$. Then
$$
 q-p = \alpha\tau v + w
$$
for some $w\in T_p^\perp\MM$.
Since both $p$ and $q$ lie within distance $\alpha^2\tau$ from $S$,
$$
 |\Pi_{S^\perp}(q-p)| \le 2\alpha^2\tau .
$$
By Corollary \ref{cor:reach2} we have
$ |w| \le \tau^{-1} |\alpha\tau v|^2 = \alpha^2\tau $,
thus
$$
 |\Pi_{S^\perp}(\alpha\tau v)| = |\Pi_{S^\perp}(q-p-w)|
 \le |\Pi_{S^\perp}(q-p)| + |w| \le 3\alpha^2\tau .
$$
The claim of the lemma follows by homogeneity.
\end{proof}

\begin{lemma}\label{l:proj-R}
There exists an absolute constant $c>0$ such that the following holds.
Assume that $\MM\subset\R^n$ has $\reach(\MM)\ge\tau$,
is $R$-exposed, and
has a $\Lambda$-Lipschitz second fundamental form where $\Lambda\ge\tau^{-2}$.
Let $S\subset\R^n$ be a linear subspace such that
\begin{equation}\label{e:assumed h bound 2}
 \sup_{x \in \MM} dist(x, S) < c \, \Lambda^{-2}R^{-4}\tau .
\end{equation}
Then $\Pi_S(\MM)$ is $2R$-exposed.
\end{lemma}

\begin{proof}
We define 
$$
 h = \sup_{x \in \MM} dist(x, S)
$$
and assume that $h$ is sufficiently small.
The required bounds for $h$ will be accumulated in the course of the proof.

Fix $p\in\MM$ and define $Y=\Pi_S(T_p\MM)$.
First we assume that
\beq\label{e:h bound 1}
h\le \alpha^2\tau
\eeq
where $\alpha\in(0,\frac14)$ is to be chosen later.
Lemma \ref{l:proj-angle} implies that 
\beq \label{e:proj angle}
 dist(x,Y) = |\Pi_{Y^\perp}(x)| \le 3\alpha|x| \qquad\text{for all $x\in T_p\MM$} .
\eeq
Since $3\alpha<1$, it follows that $\Pi_S|_{T_p\MM}$ is injective and hence
$Y$ is a $d$-dimensional linear subspace.
It is easy to see that \eqref{e:proj angle}
implies a similar property for the orthogonal complements:
\beq \label{e:proj angle perp}
 dist(z,Y^\perp) = |\Pi_{Y}(z)| 
 \le 3\alpha|z| \qquad\text{for all $z\in T_p^\perp\MM$} .
\eeq

We represent $\MM-p$ near $0$ as a graph
of a function $f\colon U\to T_p^\perp\MM$
where $U$ is the ball of radius $\tau/4$ in $T_p\MM$ centered at~0,
see Lemma \ref{l:local graph}.
This defines a local parametrization of $\MM$ given by
$$
U\ni x\mapsto \varphi(x) := p+x+f(x)
$$
We proceed in several steps.
%

\medskip
{\bf Step 1.}
We assume that
\beq \label{e:h bound 2}
 h \le \Lambda r^3
\eeq
where $r\in(0,\frac\tau4)$ is to be chosen later.
We are going to estimate the distance from $f(x)$ to $S$
for $x\in U$ such that $\|x\|\le r$.

Since  $f(0)=0$ and $d_0f=0$, we have the Taylor expansion
\begin{equation}\label{e:taylor3}
 f(x) = \frac12 d^2_0f(x,x) + \theta(x) .
\end{equation}
and an estimate
\begin{equation}\label{e:taylor3 bound}
 |\theta(x)| \le C_1\Lambda |x|^3
\end{equation}
where $C_1>0$ is an absolute constant such that
$d^2f$ is $(6C_1\Lambda)$-Lipschitz,
see Lemma \ref{l:local graph} \eqref{e:f2f lip bound}.

Since $\MM$ is contained in the $h$-neighborhood of $S$ and $p\in\MM$,
we have
$$
 | \Pi_{S^\perp}(x+f(x)) | = |\Pi_{S^\perp}(\varphi(x)-p)| \le 2h .
$$
Applying this to $x$ and $-x$ 
and summing the two inequalities we obtain
$$
 | \Pi_{S^\perp}(f(x) + f(-x)) |
 \le | \Pi_{S^\perp}(x + f(x)) | + | \Pi_{S^\perp}(-x + f(-x)) |
 \le 4h .
$$
By \eqref{e:taylor3} this can be rewritten as
$$
 | \Pi_{S^\perp}(d^2_0f(x,x) + \theta(x) + \theta(-x)) | \le 4h ,
$$
therefore, by \eqref{e:taylor3 bound},
$$
 | \Pi_{S^\perp}(d^2_0f(x,x)) | \le 4h + 2C_1\Lambda |x|^3
$$
for all $x\in U$.
Substitute $x=rv$ where $v\in T_p\MM$ is a unit vector,
and divide by $r^2$.
This yields
$$
 | \Pi_{S^\perp}(d^2_0f(v,v)) | \le 4h r^{-2} + 2C_1\Lambda r
 \le (2C_1+4)\Lambda r
$$
by \eqref{e:h bound 2}.
This holds for all unit vectors $v\in T_p\MM$,
therefore, by homogeneity,
\begin{equation}\label{e:d2f Sperp}
 | \Pi_{S^\perp}(d^2_0f(x,x)) | \le (2C_1+4)\Lambda r|x|^2
\end{equation}
for all $x\in T_p\MM$.
If $|x|\le r$ then \eqref{e:taylor3 bound} implies that
$|\theta(x)| \le C_1\Lambda r |x|^2$,
hence by \eqref{e:taylor3} and~\eqref{e:d2f Sperp},
\begin{equation}\label{e:fx perp estimate}
|\Pi_{S^\perp}(f(x)) | \le (3C_1+4)\Lambda r |x|^2 =: C_2 \Lambda r |x|^2
\end{equation}
for all $x\in U$ such that $|x|\le r$.

\medskip
{\bf Step 2.}
Let $\nu_p\in T_p^\perp\MM$ be a unit vector from Lemma \ref{l:nu_p}.
Define $w_p=\Pi_{Y^\perp}(\nu_p)$.
Our plan is to show that
\beq \label{e:2R goal}
 \langle \Pi_S(q-p), w_p\rangle \ge \frac1{4R} |\Pi_S(q-p)|^2 .
\eeq
and then apply Lemma \ref{l:nu_p} to $\Pi_S(\MM)$  
with $\frac{w_p}{|w_p|}$ in place of $\nu_p$.

In this step we handle the case when $|q-p|\le r$ where $r$
is the same as in Step~1.
In this case $q=\varphi(x)=p+x+f(x)$ for some $x\in U$ such that $|x|\le r$,
and 
\beq\label{e:proj 2R aux1}
 \langle \Pi_S(q-p),w_p\rangle
 = \langle \Pi_S(x) + \Pi_S(f(x)), w_p\rangle
 = \langle \Pi_S(f(x)), w_p\rangle
\eeq
since $\Pi_S(x)\in Y$ and $w_p\in Y^\perp$.
We rewrite the right-hand side as
$$
 \langle \Pi_S(f(x)), w_p\rangle
 = \langle f(x) - \Pi_{S^\perp}(f(x)) , \nu_p - \Pi_Y(\nu_p) \rangle
= A_0 - A_1 - A_2 + A_3
$$
where
$$
\begin{aligned}
 A_0 &= \langle f(x), \nu_p \rangle , \\
 A_1 &= \langle \Pi_{S^\perp}(f(x)), \nu_p \rangle , \\
 A_2 &= \langle f(x), \Pi_Y(\nu_p) \rangle .
\end{aligned}
$$
and $A_3 = \langle \Pi_{S^\perp}(f(x)), \Pi_Y(\nu_p) \rangle = 0$
since  $Y\subset S$. Thus 
\beq\label{e:A_0-A_1-A_2}
 \langle \Pi_S(q-p),w_p\rangle = A_0-A_1-A_2 .
\eeq
We are going to estimate $A_0$ from below and $A_1,A_2$ from above.

Since $x\in T_p\MM$ and $\nu_p\in T_p^\perp\MM$, we have 
\beq\label{e:A0 bound}
 A_0 = \langle f(x), \nu_p \rangle
 = \langle x+f(x), \nu_p \rangle
 = \langle q-p, \nu_p \rangle
 \ge \frac1{2R}|q-p|^2
\eeq
by Lemma \ref{l:nu_p}.
For $A_1$ we have
$$
 A_1 = \langle \Pi_{S^\perp}(f(x)), \nu_p \rangle \le |\Pi_{S^\perp}(f(x))|
 \le C_2 \Lambda r |x|^2
$$
by \eqref{e:fx perp estimate}.
We assume that $r$ is chosen so small that 
\beq \label{e:assumed r bound}
C_2\Lambda r\le \frac1{8R} ,
\eeq
then the previous inequality implies that
\beq\label{e:A1 bound}
 A_1 \le \frac1{8R} |x|^2 \le \frac1{8R} |q-p|^2 .
\eeq
For $A_2$ we have
$$
 A_2  = \langle f(x), \Pi_Y(\nu_p) \rangle
 = \langle \Pi_Y(f(x)), \Pi_Y(\nu_p) \rangle
 \le |\Pi_Y(f(x))|\cdot |\Pi_Y(\nu_p)| .
$$
Since $f(x)$ and $\nu_p$ belong to $T_p^\perp\MM$,
\eqref{e:proj angle perp} implies that
$$
 |\Pi_Y(f(x))| \le 3\alpha |f(x)|
$$
and
\beq\label{e:w1-w0 bound}
 |\Pi_Y(\nu_p)| \le 3\alpha |\nu_p| = 3\alpha ,
\eeq
therefore
$$
 A_2 \le 9\alpha^2|f(x)| \le 9\alpha^2 \tau^{-1} |x|^2
$$
where the last inequality follows from \eqref{e:f bound}.
We now assume that $\alpha$ is so small that
\beq \label{e:assumed alpha bound 1}
 9\alpha^2 \tau^{-1} \le \frac1{8R} ,
\eeq
then the previous inequality implies that
\beq\label{e:A2 bound}
 A_2 \le \frac1{8R}|x|^2  \le \frac1{8R} |q-p|^2 .
\eeq

Now \eqref{e:A_0-A_1-A_2} and the estimates
\eqref{e:A0 bound}, \eqref{e:A1 bound}, \eqref{e:A2 bound}
imply that
$$
\langle\Pi_S(q-p),w_p\rangle \ge \frac1{4R} |q-p|^2
\ge |\Pi_S(q-p)|^2
$$
Thus \eqref{e:2R goal} holds for all $q\in\MM$ such that $|p-q|\le r$.

\medskip
{\bf Step 3.}
Now we prove \eqref{e:2R goal} for points $q\in\MM$ such that $|q-p|\ge r$.
Recall that
$$
 \langle q-p, \nu_p \rangle \ge \frac1{2R} |q-p|^2 .
$$
by Lemma \ref{l:nu_p}.
Since $p$ and $q$ belong to the $h$-neighborhood of $S$, we have
$|\Pi_{S^\perp}(q-p)| \le 2h$, hence
$$
 \langle \Pi_{S^\perp}(q-p), \nu_p \rangle \le 2h 
 \le 2\Lambda r^3  \le \frac 1{8R} r^2 \le \frac 1{8R} |q-p|^2
$$
by \eqref{e:h bound 2}, \eqref{e:assumed r bound}
and the fact that $C_2\ge 4$, see \eqref{e:d2f Sperp}.
Subtracting this from the previous inequality we obtain
\begin{equation}\label{e:proj final 1}
 \langle \Pi_{S}(q-p), \nu_p \rangle
 =  \langle q-p, \nu_p \rangle - \langle \Pi_{S^\perp}(q-p), \nu_p \rangle
 \ge \frac3{8R} |q-p|^2 .
\end{equation}
Now we estimate $\langle \Pi_{S}(q-p), w_p-\nu_p \rangle$.
Since $w_p-\nu_p=\Pi_Y(\nu_p)$,
\begin{equation}\label{e:proj final 2}
 |\langle \Pi_{S}(q-p), w_p-\nu_p \rangle|
 \le  |\Pi_{S}(q-p)| \cdot |\Pi_Y(\nu_p)|
 \le 3\alpha  |q-p|
\end{equation}
by \eqref{e:w1-w0 bound}.
We now assume that
\beq \label{e:assumed alpha bound 2}
 \alpha \le \frac r{16R} ,
\eeq
then the previous inequality implies that
$$
 |\langle \Pi_{S}(q-p), w_p-\nu_p \rangle|
 \le \frac r{8R} |q-p| \le \frac1{8R} |q-p|^2
$$
due to the assumption $|q-p|\ge r$.
Subtracting this from \eqref{e:proj final 1} we obtain
$$
 \langle \Pi_{S}(q-p), w_p \rangle 
 = \langle \Pi_{S}(q-p), \nu_p \rangle + \langle \Pi_{S}(q-p), w_p-\nu_p \rangle
 \ge \frac1{4R}|q-p|^2 .
$$
Thus \eqref{e:2R goal} holds if $|q-p|\ge r$.

\medskip
{\bf Step 4.}
In the previous steps we have shown that \eqref{e:2R goal} holds for all $q\in\MM$.
Since $|w_p|\le 1$, it follows that the vector $w_p'=\frac{w_p}{|w_p|}$
satisfies
$$
 \langle \Pi_{S}(q)-\Pi_S(p), w_p' \rangle 
 \ge \frac1{4R}|q-p|^2 \ge  \frac1{4R}|\Pi_{S}(q)-\Pi_S(p)|^2
$$
for all $q\in\MM$.
By Lemma \ref{l:nu_p} applied to $\Pi_S(\MM)$ and $w_p'$ in place of $\nu_p$
this implies that $\Pi_S(\MM)$ is $2R$-exposed.

Now we collect the assumptions made in the course of the argument.
Those are bounds \eqref{e:h bound 1} and \eqref{e:h bound 2} on $h$,
\eqref{e:assumed alpha bound 1} and \eqref{e:assumed alpha bound 2}
on $\alpha\in(0,\frac14)$, and \eqref{e:assumed r bound} on $r\in(0,\frac\tau4)$.
Taking into account the assumption $\Lambda\ge\tau^{-2}$ and
the obvious inequality $R\ge\tau$, one sees that the required bounds
are satisfied by setting
$r=c_1\Lambda^{-1}R^{-1}$, $\alpha=c_2\Lambda^{-1}R^{-2}$s
and assuming that $h<c\Lambda^{-2}R^{-4}\tau$,
where $c_1,c_2,c>0$ are suitable absolute constants.
This finishes the proof of Lemma~\ref{l:proj-R}.
\end{proof}

\section{Principal Component Analysis and dimension reduction}

We take an adequate number of random samples, and perform Principal Component Analysis, and project the data on to a subspace $S$ of dimension $D$, as described  in \cite{filn0}. With  high probability, the image manifold $\MM$ has a Hausdorff distance of less than $c\eps$ from the original manifold $\MM_0.$ Let $K$ be the convex hull of $\MM.$



Let $S$ be an affine subspace of $\R^n$. Let $\Pi_S$ denote orthogonal projection onto $S$. Let the span of the first $d$ canonical basis vectors be denoted $\R^d$ and the span of the last $n-d$ canonical basis vectors be denoted $\R^{n-d}$. Let $\omega_d$ be the $d$ dimensional Lebesgue measure of the unit Euclidean ball in $\R^d$. 
Given $\a \in (0, 1),$ let 
\beq \label{def: beta}
\beta := \beta(\a) =  \sqrt{ \frac 1 {10} \left(\frac{\a^2\tau}{2}\right)^2\left(\frac{\a^2\tau}{4}\right)^d \left( \frac{\omega_d}{V}\right)}.\label{eq:beta1}
\eeq 

 Let \beq \label{eq:beta} D := \left\lfloor  \frac{ V}{\omega_d \beta^d}\right\rfloor + 1.\eeq
 Let 
\beq
\label{formula eps}
\eps< \beta^2/2.
\eeq
 Below, 
  \beq
\label{formula delta}
\delta< \eta/2
\eeq
  will be a small parameter that gives a bound on the probability that the conclusion $\sup_{x \in \MM} dist(x, S) < \alpha^2\tau$ in  Proposition~\ref{lem:5-May} fails. 
Choose 
\beq\label{cons:4} N_D=\lfloor C(n\sigma^2 + \sigma^2 \log(Cn\sigma^2/(\eps\de)))\sqrt{\log(C/\de)}(D/\eps^2)\rfloor,\eeq
where $C$ is a sufficiently large universal constant. 


\begin{proposition}[Proposition 3.1 \cite{filn0}] \label{lem:5-May}
Given $N_D$ data points $\{x_1, \dots, x_{N_D}\}$ drawn i.i.d from $\tmu$, let $S$ be a $D$ dimensional affine subspace that minimizes \beq \sum_{i=1}^{N_D} dist(x_i, \tS)^2, \eeq subject to the condition that $\tS$ is an affine subspace of dimension $D$,  and $\beta < c \tau$, where $\beta$ is given by (\ref{eq:beta1}).

Then, \beq \p[\sup_{x \in \MM} dist(x, S) < \alpha^2\tau] > 1 - \delta.\eeq

\end{proposition}
Recall from Subsection~\ref{subsec:pres-mod} that $\MM_0$ is a $d$-dimensional $C^{2,1}$ submanifold of $\R^n$. Let $K_0$ be the convex hull of $\MM_0$. We assume that  $$ \MM_0 \subseteq \partial K_0 \subseteq B^n_1(0). $$  We also suppose that $\MM_0$ has volume ($d$-dimensional Hausdorff measure) less or equal to $V$, reach (i.e. normal injectivity radius) greater or equal to $\tau$. 

\begin{definition}

After a suitable orthogonal change of coordinates, we identify $S$ with $\R^D$. \ben
\item Let $\MM := \Pi_S \MM_0$.
\item Let $K $ be the convex hull of $\MM$.
\een
\end{definition}

\section{Continuity of outward normal cones}

\begin{theorem}\lab{thm:stable}
Let $\MM=\MM^d\subset\R^n$ be closed a $C^{2,1}$ submanifold satisfying
the above geometric bounds, and let $K=\conv(\MM)$.
Then the outer normal cone $N_K(x)$ is a Lipschitz function of $x\in \MM$:
$$
 d_{CH}(N_K(x),N_K(y)) \le L\|x-y\|
$$
for all $x,y\in \MM$, where $L>0$ is determined by the parameters
$R,\tau,\Lambda$ from the geometric bounds. In fact, $CR(\Lambda+\tau^{-2})$.
\end{theorem}
\begin{proof}

The next two lemmas are from convex geometry.
The first one allows us to estimate the distances between
tangent cones instead of outer normal cones,
using the fact that $N_K(p)=T_K(p)^\circ$.

\begin{lemma} \label{l:polar cone distance}
For any closed convex cones $K_1,K_2\subset\R^n$,
$$
 d_{CH}(K_1^\circ,K_2^\circ) = d_{CH}(K_1,K_2) .
$$
\end{lemma}

\begin{proof}
Denote $s=d_{CH}(K_1,K_2)$; we may assume that $s>0$.
W.l.o.g. the distance $s$ is realized by points $x_0\in K_1\cap B_1(0)$
and $y_0\in K_2\cap B_1(0)$ such that $y_0$ 
is a nearest point to $x_0$ in $K_2\cap B_1(0)$
and $\|x_0-y_0\|=s$.
Since $K_1$ and $K_2$ are cones, we have $\|x_0\|=1$.
Define $v=\frac{x_0-y_0}{\|x_0-y_0\|}=\frac{x_0-y_0}{s}$.

Observe that $y_0$ is a nearest point to $x_0$ in
the whole cone~$K_2$.
This implies that $\langle x_0-y_0, y_0\rangle = 0$.
Therefore
$$
 \langle x_0-y_0, x_0\rangle = \langle x_0-y_0, x_0-y_0\rangle = s^2
$$
and
$$
 \langle v, x_0\rangle = s .
$$
Since $y_0$ is nearest to $x_0$ in $K_2$ and $K$ is convex, 
for every $y\in K_2$ we have
$$
 \langle y,x_0-y_0 \rangle = \langle y-y_0,x_0-y_0 \rangle \le 0 .
$$
Therefore $v\in K_2^\circ$.
Consider an arbitrary $w\in K_1^\circ$ and estimate
$\|v-w\|$ from below as follows:
First observe that $\langle w,x_0\rangle \le 0$
since $x_0\in K_1$ and $w\in K_1^\circ$. Then
$$
 \|v-w\| \ge \langle v-w, x_0\rangle
 = \langle v,x_0\rangle - \langle w,x_0\rangle \ge s
$$
since $\langle v, x_0\rangle = s$ and $\langle w,x_0\rangle \le 0$.
(The last inequality follows from the facts that
$x_0\in K_1$ and $w\in K_1^\circ$).

Thus $\|v-w\| \ge s$ for all  $w\in K_1^\circ$. Since $v\in K_1^\circ$
and $\|v\|=1$, this implies that
$$
d_{CH}(K_1^\circ,K_2^\circ) \ge s = d_{CH}(K_1,K_2) .
$$
The reverse inequality follows from the same arguments applied
to $K_1^\circ$ and $K_2^\circ$ in place of $K_1$ and~$K_2$,
and the Bipolar Theorem.
\end{proof}

The next lemma estimates the distance between convex cones
generated by subsets of $\R^n$.
To get a sensible estimate we have to assume that
each set is strictly separated away from 0 by some hyperplane,
this is controlled by the parameter $\alpha.$

\begin{lemma}\label{l:cone base distance}
Let $\alpha,\delta>0$ and let $X_1,X_2\subset\R^n$ be such that
\begin{enumerate}
\item 
There exist unit vectors
$v_1,v_2\in\R^n$ such that $\langle x,v_i\rangle\ge\alpha$
for all $x\in X_i$, $i=1,2$.
\item
For every $x\in X_1$ there exists $y\in\cone(X_2)$ such that
$\|x-y\|\le\delta$,
and the same holds with $X_1$ and $X_2$ interchanged.
\end{enumerate}
Then
$$
 d_{CH}(\conv\cone(X_1),\conv\cone(X_2)) \le \alpha^{-1}\delta .
$$
\end{lemma}

\begin{proof}
Let $Y_i = \conv(X_i)$, $i=1,2$.
It is easy to see that both assumptions are preserved if $X_i$
is replaced by $Y_i$, $i=1,2$.
The first assumption implies that $\|x\|\ge\alpha$
for all $x\in Y_1\cup Y_2$.
Note that $\conv\cone(X_i)=\cone(Y_i)$.

Every point from $\cone(Y_1)\cap B_1(0)$ can be written as $sx$
for some $x\in Y_1$ and $s\in[0,\alpha^{-1}]$.
By assumptions there exists $y\in\cone(Y_2)$ such that
$\|x-y\|\le\delta$. Moreover $y$ can be chosen so that $\|y\|\le\|x\|$
(replace $y$ by the point nearest to $x$ on the ray
$\{ty:y\ge 0\}$). Then $\|sy\|\le\|sx\|\le 1$.

Now for any point $sx\in \cone(Y_1)\cap B_1(0)$,
where $x\in Y_1$ and $s\in[0,\alpha^{-1}]$,
we have a point $sy\in\cone(Y_2)\cap B_1(0)$
with $\|sx-sy\|= s\|x-y\|\le\alpha^{-1}\delta$.
The same holds with indices 1 and 2 exchanged,
hence the desired inequality holds.
\end{proof}

Fix $p\in \MM$ and let $K=\conv(\MM)$.
We are interested in the tangent cone 
$$
 T_K(p)= \cl\conv\cone(\MM-p) .
$$
In order to be able to apply Lemma \ref{l:cone base distance},
we split off a linear part $T_p\MM$ of this cone.
Namely let $\Pi_p^\perp$ denote the orthogonal projection from $\R^n$
to $T_p^\perp\MM$, then
$$
 T_K(p) = T_p\MM + T_K^\perp(p).
$$
where
\beq\label{e:TKperp}
 T_K^\perp(p) = T_K(p) \cap T_p^\perp\MM =
 \cl\conv\cone(\Pi_p^\perp(\MM-p)) .
\eeq
Now with Lemma \ref{l:polar cone distance} it suffices
to prove that $T_K^\perp(p)$ is a Lipschitz function of~$p$
(since $T_p\MM$ is a Lipschitz function of $p$ with Lipschitz
constant $\tau^{-1}$).

Define a map $\Phi_p\colon \MM\setminus\{p\}\to T_p^\perp\MM$ by
$$
 \Phi_p(x) = \frac{\Pi_p^\perp(x-p)}{\|x-p\|^2}
$$
and let $X_p$ be the image of $\Phi_p$:
$$
 X_p = \{ \Phi_p(x) \mid x\in \MM\setminus\{p\} \} .
$$
From \eqref{e:TKperp} we have
\beq
 T_K^\perp(p) = \cl\conv\cone(X_p)
\eeq
Let $\nu_p\in T_p^\perp\MM$ be a unit vector from Lemma \ref{l:nu_p},
then
$$
 \langle \Phi_p(x), \nu_p \rangle = \|x-p\|^{-2} \langle x-p,\nu_p \rangle
 \ge \frac1{2R}
$$
for all $x\in M\setminus\{p\}$,
where the last inequality follows from Lemma \ref{l:nu_p}.
Hence $X_p$ satisfies the first assumption of
Lemma \ref{l:cone base distance} with 
\beq\label{def: alpha}
\alpha=\frac1{2R}.
\eeq

Our plan is to compare $X_p$ and $X_q$ for two nearby points $p,q\in M$
and apply Lemma \ref{l:cone base distance} to estimate the distance
between $T_K^\perp(p)$ and $T_K^\perp(q)$.
We assume that $p$ and $q$ are close to each other,
say $\|q-p\|\le \tau/10$.
By Lemma \ref{l:local graph}, $(\MM-p)\cap B_\tau(0)$ is a graph
of a function $f\colon U\subset T_p\MM \to T_p^\perp\MM$.
Define a local parametrization $\phi\colon U\to M$ by
$$
 \phi(x) = p + x + f(x), \qquad x\in U .
$$
We have $p=\phi(0)$ and $q=\phi(\bar q)$ for some $\bar q\in U$
with $\|\bar q\|\le \|q-p\|$.

Pick $x\in \MM\setminus\{p\}$. Our goal is to find $y\in \MM\setminus\{q\}$
and $\lambda>0$ such that 
\begin{equation} \label{e:cone goal}
\|\Phi_p(x)-\lambda\Phi_q(y)\| \le L_1 \|p-q\|
\end{equation}
for some $L_1>0$ determined by the geometric bounds.
Then Lemma \ref{l:cone base distance} applied to $X_p$ and $X_q$
will imply that
$$
d_{CH}(T_K^\perp(p),T_K^\perp(q)) \le 2RL_1 \|p-q\| .
$$

We consider two cases, one is when $x$ is ``near'' $p$
and the other is when $x$ is ``far away'' from $q$.

\medskip
{\bf Case 1: $\|x-p\|\le \tau/4$}.
In this case $x=\phi(\bar x)$ for some $\bar x\in U$, $\|\bar x\|\le\tau/4$.
We are going to use $y=\phi(\bar q+\bar x)$ for \eqref{e:cone goal}.
We have
$$
 x - p =  \phi(\bar x)-\phi(0) = \bar x + f(\bar x),
$$
hence
$$
 \Pi_p^\perp(x-p) = \Pi_p^\perp(f(\bar x)) =  f(\bar x),
$$
since $\bar x\in T_p\MM$ and $f(\bar x)\in T_p^\perp\MM$, and
$$
 y - q = \phi(\bar q+\bar x) - \phi(\bar q) =
 \bar x + f(\bar q+\bar x) - f(\bar q) ,
$$
hence
$$
 \Pi_q^\perp(y-q) 
 = \Pi_q^\perp(f(\bar q+\bar x) - f(\bar q) -  d_{\bar q}f(\bar x))
$$
since $\bar x + d_{\bar q}f(\bar x) \in T_qM$.
By Lemma \ref{l:C3 bound implication},
$$
 \|f(\bar q+\bar x) - f(\bar q) - d_{\bar q}f(\bar x) - f(\bar x)\|
 \le C\Lambda \|\bar x\|^2\|\bar q\| ,
$$
therefore
\begin{equation}\label{e:cone1}
\| \Pi_q^\perp(y-q) - \Pi_q^\perp (f(\bar x)) \|
\le C\Lambda \|\bar x\|^2\|\bar q\| .
\end{equation}

Next we estimate the difference between $\Pi_q^\perp (f(\bar x))$
and $\Pi_p^\perp(x-p) = f(\bar x)$.
Since the second fundamental form of $\MM$ is bounded by $\tau^{-1}$,
the second differential of $f$ is bounded by $C\tau^{-1}$,
hence
$$
 \|f(\bar x)\| \le C\tau^{-1}\|\bar x\|^2 .
$$
The second fundamental form bound also implies that
$\|\Pi_q^\perp-\Pi_p^\perp\|\le C\tau^{-1}\|q-p\|$, hence
$$
 \|\Pi_q^\perp(f(\bar x)) - \Pi_p^\perp(x-p))\|
 = \|\Pi_q^\perp(f(\bar x)) - f(\bar x)\| \le C\tau^{-1}\|f(\bar x)\|
 \le C\tau^{-2}\|\bar x\|^2\|\bar q\|  .
$$
Summing this with \eqref{e:cone1} we obtain
$$
 \| \Pi_q^\perp(y-q) - \Pi_q^\perp (x-p) \| 
 \le C(\Lambda+\tau^{-2})\|\bar x\|^2\|\bar q\|
 \le C(\Lambda+\tau^{-2}) \|x-p\|^2\|q-p\|.
$$
Dividing this inequality by $\|x-p\|^2$ yields that
$$
 \| \lambda \Phi_q(y) - \Phi_p(x)\| \le C(\Lambda+\tau^{-2})\|q-p\|
$$
where $\lambda = \|y-q\|^2/\|x-p\|^2$.
This is a desired bound of type \eqref{e:cone goal}.

\medskip

{\bf Case 2: $\|x-p\|>\tau/4$}.
In this case we prove \eqref{e:cone goal} for $y=x$.
Here we only need the identity
$$
 \|(x-p)-(x-q)\| = \|p-q\|
$$
and the bound $\|\Pi_q^\perp-\Pi_p^\perp\|\le C\tau^{-1}\|q-p\|$.
For $\lambda = \|y-q\|^2/\|x-p\|^2$ we have
$$
 \|\Phi_p(x) - \lambda\Phi_q(x)\| 
 \le \frac{\|\Pi_p^\perp(q-p)\| + \|\Pi_q(x-q)-\Pi_p(x-q)\|}{\|x-p\|^2} 
 \le C\tau^{-2} \|q-p\| .
$$
This finishes Case 2.

\medskip

In both cases we obtained \eqref{e:cone goal} with $L_1$
bounded by $C(\Lambda+\tau^{-2})$.
By Lemma \ref{l:cone base distance} this implies that $T_K^\perp(p)$
is Lipschitz in $p$ with Lipschitz constant $CR(\Lambda+\tau^{-2})$.
With Lipschitz continuity of $T_p\MM$ it follows that  $T_K(p)$
is Lipschitz in $p$, and then Lemma \ref{l:polar cone distance}
finishes the proof of Theorem~\ref{thm:stable}.
\end{proof}

\section{Weak oracles}

The contents of this section closely follow the book  \cite{Lov}. Let
$S(K, \eps)$ be the set of all points within $\eps$ of $K$. Moreover, let $S(K, -\eps)$ be the set of all points $p$ in $K$ that are not contained in an $\eps$-neighborhood of $\R^D\setminus K.$

\begin{definition}[Weak Optimization Problem, WOPT]\lab{def:wopt}
Given a vector $b \in \Q^D,$ and a rational number $\eps > 0$, either
\ben
\item find a vector $y \in \Q^D$ such that $y \in S(K, \eps)$ and   $\langle b,x\rangle \leq \langle b,y \rangle   + \eps$ for all $x \in S(K, -\eps),$ or
\item assert that $S(K, -\eps)$ is empty.
\een
\end{definition}




\begin{definition}[Weak Validity Problem, WVAL]\lab{def:wval}
Given a vector $b \in \Q^D,$ a rational number ${{t}}$, and a rational number $\eps > 0$, either
\ben
\item assert that $\langle b,x \rangle \leq {{t}} + \eps$ for all $x \in S(K, -\eps),$ or
\item assert that $\langle b, x\rangle  \geq {{t}} - \eps$ for some $x \in S(K, \eps)$\\(i.\,e., $b^Tx \leq {{t}}$ is almost invalid).
\een
\end{definition}

\begin{definition}[Weak Separation Problem, WSEP]
Given a vector $y \in \Q^D$ and a rational number $\de > 0,$ either
\ben \item  assert that $y \in S(K, \de),$ or
\item  find a vector $b  \in \Q^D$  with $\|b\|_\infty = 1$ such that $\langle b, x\rangle  \leq \langle  b, y\rangle  + \de$
for every $x \in S(K, -\de)$.
(i. e., find an almost separating hyperplane).
\een
\end{definition}

\subsection{Lemmas relating the oracles}
We combine our techniques with fundamental optimization results based on the ellipsoid algorithm in \cite{Lov}
that gives an estimate for the support function $$s(b)=\sup_{x\in \MM} b\cdot x,$$ 
where  $b  \in \R^D$, $\|b\|_{ \R^D}=1$.

We say that a convex body $K' \subseteq \R^{D'}$ is $R'$ circumscribed if $K'$ is contained in the ball of radius $R'$ in $\R^{D'}$ centered at the origin. 
We say that an algorithm for solving  WOPT runs in oracle-polynomial time in $(K', D', R')$ given access to an oracle to the solution to WSEP for if there is an algorithm that runs in time polynomial in $D', R'$ and the encoding length (as rational numbers or vectors) of $b$ and $\eps$ that uses arithmetic operations or calls to solutions of WVAL.
Analogous terminology is used if WOPT is replaced by WSEP and WSEP by WVAL as happens below.\footnote{{Note that on page 102 of \cite{Lov}, the definition of oracle-polynomial-time carries a dependence on $\langle K \rangle$, the encoding length of the convex set $K$. However, it can be seen that in the theorems we refer to, namely Theorem 4.4.4 and Corollary 4.2.7 of \cite{Lov}, there is no dependence on the encoding length of the convex set $K$; access to the appropriate oracle is sufficient.}}
\begin{proposition}\lab{lem:2.1}
There is $\eps_0>0$ such that the following holds:
Assume that $0<\eps<\eps_0$, $0<\eta<\frac 12,$ and 
    \beq   \label{ineq for N claim}
N\ge\bigg( \frac 
{3\cdot 2^{10}\sqrt{2\pi}\sigma^2}{\delta^2} \cdot
 \frac{V}{(c\sqrt{{\eps}}\tau/4)^d \omega_d }\bigg)^3 \exp\left( \left(\frac{\sigma}{{\eps} \tau/4} \log ( \frac{{V}}{(c\sqrt{\eps}\tau/4)^d \omega_d })\right)^2\right)\log(\eta^{-1}).
 \eeq
Then the algorithm Find-Distance, given below, with the inputed parameters
$\eps>0$ and  $b  \in \R^D$, $\|b\|_{ \R^D}=1$ and 
points $X_1,\dots, X_N$, sampled independently from the distribution $\mu*G_\sigma^{(n)}$, gives value $s^{est}(b)$ which satisfies
\beq
|s^{est}(b)-s(b)|<\eps
\eeq
with probability  larger than $ 1 - \eta$. The number of arithmetic operations in the algorithm Find-Distance 
is polynomial in $N$
\end{proposition}

The proof of Proposition \ref{lem:2.1} is given in Section~\ref{sec:3}.

\begin{algorithm}[H]
{\bf Algorithm Find-Distance}  

{Input  parameters: $\eps>0$, $D,N\in \mathbb Z_+$ and  $b  \in \R^D$, $\|b\|_{ \R^D}=1$
and the sample points $X_1,X_2,\dots, X_N\in \R^D$.}

\ben 
\item Let $\delta=\eps/16$ and 
$$
  \Gamma_\delta  := \frac{(c\sqrt\delta\tau)^d \omega_d }{{V}}   (\sqrt{2\pi}\sigma)^{-1}\exp\left(- \frac{1}{2} \left(\frac{\sigma}{\delta \tau} \log(\frac{V}{(c\sqrt{\delta}\tau)^d \omega_d }) \right)^2\right),
  \quad
  r_\delta:=  \frac{\sigma^2}{{\delta} \tau} \log (\frac{{V}}{(c\sqrt\delta\tau)^d \omega_d }) 
  $$

\item Let $j_-=\lfloor\frac 1\delta(r_\delta- 1- {\delta}\tau)\rfloor-1$ and $j_+=\lfloor\frac 1\delta(r_\delta+ 1+ {\delta}\tau)\rfloor+1$. {Set $j=j_-$}.


\item Repeat

\ben

\item Increase $j$ by one

\item Compute $\Gamma^{est}_{ j\delta,b}=\frac{1}{N\delta }\cdot \#\{i=1,\dots,N\ :\  j\delta<X_i\cdot b\le ( j+1)\delta \}$, where  $\#$ denotes the cardinality of a set.

\een

\noindent Until  $\Gamma^{est}_{ j\delta,b}\le \Gamma_\delta$ or $j=j_+$

\item If  $j<j_+$ and then output $s^{est}(b)=j_*\delta-r_\delta$, otherwise output that the algorithm has failed.

\een
\end{algorithm}

Using  algorithm Find-Distance it 
is possible to construct a weak validity oracle
which for the input $\eps>0$, $t\in \R$ and  $b  \in \R^D$, $\|b\|_{ \R^D}=1$ gives a correct answer (i.e., asserts
that either 1 or 2 in Definition \ref{def:wval} holds)
 with probability larger than $ 1 - \eta$  {when a number of random samples $N$ 
in Proposition \ref{lem:2.1} 
satisfies}
 \beq\label{bound for N} N\ge \, \tilde{\Omega}\left(\exp\left(\left(\frac{\sigma}{\eps\tau}\log \frac{V}{\tau^d}\right)^2\right) \log \left( \eta^{-1}\right)\right).  \eeq


\begin{lemma}\lab{lem: before 5.4}
It is possible to solve a Weak Separation Problem (WSEP) for $K$ in oracle-polynomial time, given access to a weak validity oracle.
\end{lemma}
\begin{proof}
This follows from Theorem (4.4.4)  in \cite{Lov}: There exists an oracle-polynomial time algorithm that solves the weak separation problem for every circumscribed convex body $(K; D, 1)$ given by a weak validity oracle.
\end{proof}


\begin{lemma}\lab{lem:5.4}
It is possible to solve a Weak Optimization Problem (WOPT) for $K$ in oracle-polynomial time, given access to a weak separation oracle.
\end{lemma}
\begin{proof}
This follows from Corollary (4.2.7)  in \cite{Lov}: There exists an oracle-polynomial time algorithm that solves the weak optimization problem for every circumscribed convex body $(K; D, 1)$ given by a weak separation oracle. 
\end{proof}

\begin{algorithm}[H]
{{\bf Algorithm  Weak-Optimization-Oracle} 

Input  parameters: $\eps>0$, $D,N,p\in \mathbb Z_+$ and  $b  \in \R^D$, $\|b\|_{ \R^D}=1$
and the sample points $X_1,X_2,\dots, X_{N_1}\in \R^D$, where $N_1=pN$.

\ben 
\item
Solve the  weak optimization problem with parameters $\eps$ and $b$ using the procedure given in Theorem (4.4.4) and Corollary (4.2.7)  in \cite{Lov}. In this procedure,  use
the algorithm Find-Distance with the sample points $\{X_{Nj-N+1},\dots,X_{Nj}\}$, where
$j=1,2,\dots,p$, to
implement the weak validity oracle {$p$ times.}

 \item If in step 1 none of the  algorithms Find-Distance fail and 
   the solution of the  weak optimization problem is found in using $p$ calls to the weak validity oracle, output the solution of the
    weak optimization problem. Otherwise, output that the algorithm has failed.

\een}
\end{algorithm}

{Observe that by  Lemma \ref{lem: before 5.4} and \ref{lem:5.4}, the number $p$ of
the times needed to call the  algorithm Find-Distance 
can be chosen polynomially to depend on} {{$D$ and $\log \frac{1}{\eps}$.}}
\section{Computational solution to the Weak Validity Problem (WVAL)}\lab{sec:3}

\begin{figure}
    \centering
        \includegraphics[width=2.5in]{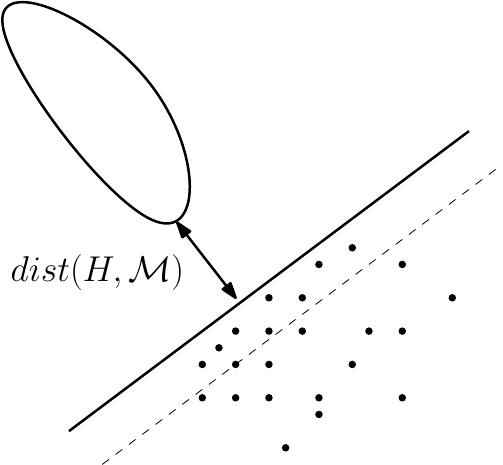} 
        \caption{Estimating the distance of $H$ from $\MM$ by counting the number of samples on the side of $H$ that does not contain $\MM$.} \label{fig:distHM}
    \end{figure}

\begin{proof}[Proof of Proposition~\ref{lem:2.1}] { As $\MM\subset B^D_1(0)$, we may assume that $\tau\le 1.$ Assume that $0<\eps<\min(\frac 1{16},\frac{\sigma}{\tau},\frac{\sigma^2 }{\tau })$,  and denote $\delta=\eps/16$.}
Consider a codimension one hyperplane { $H\subset \R^D$} at a distance greater than $1$ from the origin.
Recall that $\MM \subseteq B^D_1(0)$. Our goal here is to find an $\eps$-accurate estimate of $\dist(H, \MM)$ with as few computational steps as possible.

We  use $\rho$ to denote the density given by 
$$
\rho(y):= \int_\MM (\sqrt{2\pi}\sigma)^{-D}\exp\left(- \frac{ |y - x|^2}{2\sigma^2}\right){ \frac 1{|\MM|}}\la^d_\MM(dx),
$$ where $y \in \R^D$ and $\la^d_\MM$ is the volume measure on $\MM$.

Let us denote by $\dist(H, x)$, the $\ell_2$ distance between the  $x$ and the nearest point $y$, where $y \in H$.
Let us denote by $\dist(H, \MM)$, the $\ell_2$ distance between the two nearest points  $x$ and $y,$ where $x\in \MM$ and $y \in H$.
Let \beqs \Gamma(H):= \int_H \rho(y)\la^{n-1}_H(dy)\eeqs 
and
 $$
H_{{{t}}, b} = \{ { y\in \R^D\,}|\, \langle b,y\rangle = {{t}}\}, 
$$
where $b\in \R^D$, $\|b\|_2 = 1$ and ${{t}} \geq 1$. 
We denote
\beq 
& &\kappa_0 := \frac{1}{(\sqrt{2\pi}\sigma)^{-1}},\lab{eq:k0}\\
 & &\kappa_1 := \frac{{V}}{(c\sqrt\delta\tau)^d \omega_d }\frac{1}{  (\sqrt{2\pi}\sigma)^{-1}}.\lab{eq:k1}
 \eeq 
 where we use $c=1$.

\begin{lemma}\lab{lem:2.1-12oct}
The function $\gamma\to \Gamma(H_{\gamma, b})$ is a strictly decreasing for $\gamma \ge 1$.
Moreover, { the distance of the manifold $\MM$ and the affine hyperplane $H=H_{\gamma, b}$, $\gamma \ge 1$, satisfies}
\beqs - {\delta}\tau +\sqrt{ (2 \sigma^2)\log \left(({\Gamma(H)}{\kappa_1})^{-1}\right)}  \leq    \dist(H, \MM) \leq  \sqrt{(2 \sigma^2)\log \left(({\Gamma(H)}{\kappa_0})^{-1}\right)}.\eeqs
\end{lemma}
\begin{proof}
We integrate along  $H$, and see that
\beq\label{formula for Gamma}
\Gamma(H) =  \int_H \rho(y)\la^{D-1}_H(dy) = \int_\MM (\sqrt{2\pi}\sigma)^{-1}\exp\left(- \frac{ \dist(x,H)^2}{2\sigma^2}\right)
\frac {\la^d_\MM(dx)}{{ |\MM|}},
\eeq 
where  $\la^{D-1}_H$ is the $D-1$ dimensional Lebesgue measure on $H$.
As $\MM\subset B^D_1(0)$, for all $x\in \MM$ the function $\gamma\to \dist(x,H_{\gamma, b})$ is a strictly decreasing for $\gamma \ge 1$. This implies that the function $\gamma\to \Gamma(H_{\gamma, b})$ is a strictly decreasing for $\gamma \ge 1$.

We observe that 
\beq 
\int_\MM (\sqrt{2\pi}\sigma)^{-1}\exp\left(- \frac{ \dist(H, x)^2}{2\sigma^2}\right)\frac {\la^d_\MM(dx)}{{ |\MM|}}
\leq  \int_\MM (\sqrt{2\pi}\sigma)^{-1}\exp\left(- \frac{ \dist(H, \MM)^2}{2\sigma^2}\right)\frac {\la^d_\MM(dx)}{{ |\MM|}}
,\lab{eq:7}\eeq 
and therefore, that 
\beqs \exp\left(- \frac{ \dist(H, \MM)^2}{2\sigma^2}\right) & \geq &  \frac{\int_\MM (\sqrt{2\pi}\sigma)^{-1}\exp\left(- \frac{ \dist(H, x)^2}{2\sigma^2}\right)\la^d_\MM(dx)}{|\MM|(\sqrt{2\pi}\sigma)^{-1}} =  \frac{\Gamma(H)}{(\sqrt{2\pi}\sigma)^{-1}} .\eeqs
We then see that  \beq \dist(H, \MM) \leq  \sqrt{(- 2 \sigma^2)\log \left({\Gamma(H)}{\kappa_0}\right)}.\lab{eq:10}\eeq

Recall that $\MM \in \G(d, V, \tau)$. For $x \in \MM$ denote  the orthogonal projection from $\R^D$
 to the affine subspace tangent to $\MM$ at $x$,  $Tan(x)$ by $\Pi_{x}$. 

\begin{lemma}[Lemma 12, \cite{putative}]\lab{cl:g1sept}
Suppose that $\MM \in \G(d, V, \tau)$. Let \beqs U:={\left \{y\in \R^D\,\big||y-\Pi_xy| \leq \tau/4\right\} \cap \left  \{y\in \R^D\,\big||x-\Pi_xy| \leq \tau/4\right\}.}\eeqs 
Then, $$\Pi_x(U \cap \MM) = \Pi_x(U).$$
\end{lemma}

 As we have assumed  that $0<\eps<1/16$ and $\delta=\eps/16$, we see using 
 formulas \eqref{eq:k0} and \eqref{eq:k1} and Lemma \ref{cl:g1sept} that
 $\Pi_x(B_x(4\sqrt\delta\tau)) \subset \Pi_x(U \cap \MM)$ and hence ${(4\sqrt\delta\tau)^d \omega_d }\leq V$. Thus,  when we use
$c=1$ in \eqref{eq:k1}  to define $\kappa_1$, we have 
  \beq 
  \kappa_1 \geq e \kappa_0.\lab{eq:e-kappa}
 \eeq

Let $\MM_\delta$ denote the set of points in $\MM$ whose distance from $H$ is less or equal to $\dist(H, \MM) + \delta \tau,$ and let $x$ be a nearest point on $\MM$ to $H$. As $\delta<1/16$, we see using the reach condition, (see Corollary \ref{cor:reach condition} and Lemma~\ref{lem:g1}) 
 that $\Pi_x \MM_{\delta}$ contains a $d$-dimensional ball of radius greater or equal to $\sqrt{\delta}\tau$,
see also  Lemma 12 of \cite{putative}.
Similarly to (\ref{eq:7}), 
\beqs 
\Gamma(H) &=& \int_\MM (\sqrt{2\pi}\sigma)^{-1}\exp\left(- \frac{ \dist(H, x)^2}{2\sigma^2}\right)\frac 1{{ |\MM|}}\la^d_\MM(dx) 
\\& \geq &  \frac 1{{ |\MM|}} \int_{\MM_{\delta}} (\sqrt{2\pi}\sigma)^{-1}\exp\left(- \frac{ ({\delta}\tau + \dist(H, \MM))^2}{2\sigma^2}\right)\la^d_\MM(dx)
\\
& \geq & \frac 1{{ V}}
(c\sqrt\delta\tau)^d \omega_d  (\sqrt{2\pi}\sigma)^{-1}\exp\left(- \frac{ ({\delta}\tau + \dist(H, \MM))^2}{2\sigma^2}\right).\eeqs 
It follows that 
\beqs
 \exp\left(- \frac{ ({\delta}\tau + \dist(H, \MM))^2}{2\sigma^2}\right) & \leq &
 \frac{{ V}}{(c\sqrt\delta\tau)^d \omega_d }
  \frac{1}{(\sqrt{2\pi}\sigma)^{-1}}\Gamma(H)=\kappa_1\Gamma(H).
  \eeqs
We then see that  \beq ({\delta}\tau + \dist(H, \MM)) \geq \sqrt{ (- 2 \sigma^2)\log \left({\Gamma(H)}{\kappa_1}\right)},\lab{eq:52}\eeq and putting this together with (\ref{eq:10}), we obtain 
 \beqs - {\delta}\tau +\sqrt{ (2 \sigma^2)\log \left(({\Gamma(H)}{\kappa_1})^{-1}\right)}  \leq    \dist(H, \MM) \leq  \sqrt{(2 \sigma^2)\log \left(({\Gamma(H)}{\kappa_0})^{-1}\right)}.\eeqs
 This proves Lemma \ref{lem:2.1-12oct}.
 \end{proof}

 {Motivated by Lemma \ref{lem:2.1-12oct}, we pose the following definition}
 
 \begin{definition} We define $Gap(H, \MM)$ by 
 \beqs 
 Gap(H, \MM) := \sqrt{(- 2 \sigma^2)\log \left(({\Gamma(H)}{\kappa_0})^{-1}\right)} - \left(- {\delta}\tau +\sqrt{ (- 2 \sigma^2)\log \left(({\Gamma(H)}{\kappa_1})^{-1}\right)}\right).\eeqs
\end{definition}

We will need an upper bound on $Gap(H, \MM).$

\begin{lemma}\lab{lem:2.1new}
Suppose that { the hyperplane $H$ satisfies}
\beqs  \sigma \log \frac{\kappa_1}{\kappa_0} \leq  {\delta}\tau \sqrt{2\log  \left( (\Gamma(H) \kappa_1)^{-1}\right)}.\eeqs 
Then, 
\beqs Gap(H, \MM) \le  2\tau{\delta}. \eeqs 
\end{lemma}
\begin{proof}We denote $\Gamma(H)$ by $\Gamma.$
Observe that \beqs  Gap(H, \MM) & = & \sqrt{(- 2 \sigma^2)\log \left(({\Gamma}{\kappa_0}\right)} - \left(- {\delta}\tau +\sqrt{ (- 2 \sigma^2)\log \left(({\Gamma}{\kappa_1}\right)}\right)\\
& = &
{\delta} \tau + \left( \frac{(- 2 \sigma^2)\log \left(({\Gamma}{\kappa_0}\right) -  (- 2 \sigma^2)\log \left(({\Gamma}{\kappa_1}\right)} { \sqrt{ (- 2 \sigma^2)\log \left(({\Gamma}{\kappa_0}\right)} +                                              \sqrt{ (- 2 \sigma^2)\log \left(({\Gamma}{\kappa_1}\right)} }\right).\eeqs

This can be simplified as follows.
\beqs 
& &{\delta} \tau +\sqrt{2}\sigma \left( \frac{\log \left(({\Gamma}{\kappa_0})^{-1}\right) - \log \left(({\Gamma}{\kappa_1})^{-1}\right)} {\sqrt{\log \left(({\Gamma}{\kappa_0})^{-1}\right)} + \sqrt{\log \left(({\Gamma}{\kappa_1})^{-1}\right)} }\right) \\
 & \leq & {\delta} \tau +\sqrt{2}\sigma  \left( \frac {\log \left(\kappa_1/\kappa_0\right)}{\sqrt{\log \left(({\Gamma}{\kappa_0})^{-1}\right)} + \sqrt{\log \left(({\Gamma}{\kappa_1})^{-1}\right)}} \right) \\
 & \leq &{{\delta} \tau} +\sqrt{2}\sigma  \frac{\left(\log \frac{\kappa_1}{\kappa_0}\right)} {2\sqrt{\log \left(({\Gamma}{\kappa_1})^{-1}\right)} } . \eeqs 
In order to make $Gap(H, \MM)$ less than $2\tau {\delta}$, it suffices to have 
\beqs  \sigma \log \frac{\kappa_1}{\kappa_0} \leq {\delta}\tau \sqrt{ 2 \log \left((\Gamma \kappa_1)^{-1}\right)}.\eeqs 

\end{proof}

\subsection{Estimating $\dist(H_{\gamma, b}, \MM)$ accurately.}

 {We have assumed that  ${\delta} < \min(\frac 1{16^2},\frac{\sigma}{16\tau})$. 
Let us} define $\Gamma_{\delta}$ as the solution of the equation
\beq 
 \sigma \log \frac{\kappa_1}{\kappa_0} = {\delta}\tau \sqrt{ 2 \log \left((\Gamma_{\delta} \kappa_1)^{-1}\right)}.
 \lab{eq:27.1}
 \eeq 
In other words,
\beq 
 \Gamma_{\delta}  = {}  \kappa_1^{-1}\exp\left(- \frac{1}{2} \left(\frac{\sigma}{{\delta} \tau} \log \frac{\kappa_1}{\kappa_0}\right)^2\right). \lab{eq:71}
  \eeq
  Note that as   ${\delta} \le \frac{\sigma}{16\tau}$, we have $\Gamma_{\delta}  \le e^{-8}\kappa_1^{-1}$.
  We also observe that $2\Gamma_{\delta}<e/\kappa_1$.

Let $\gamma({\delta})$ be the unique real number that satisfies the equation
 \beq
\label{def of gamma eps}
 \Gamma(H_{\gamma({\delta}), b}) = \Gamma_{\delta}.
 \eeq

Let us the function $\ell:(0,e/\kappa_1]\to \R$, given by
$$
\ell(\Gamma)= \sqrt{ (2 \sigma^2)\log \left(({\Gamma}{\kappa_1})^{-1}\right)}.
$$

\begin{lemma}\label{lem: dist estimate}
It holds that 
\beq\label{dist estimate}
|\ell \left(\Gamma(H_{\gamma', b})\right)-\dist(H_{\gamma', b}, \MM)| 
\leq {\delta}\tau 
\eeq
for all $\gamma' \ge \gamma({\delta}).$
\end{lemma}

\begin{proof}
Using Lemma \ref{lem:2.1new} and formula  \eqref{eq:71} we see that  for any $\gamma' \ge  \gamma({\delta}),$
\beq \label{eq on on Gap}
Gap(H_{\gamma', b}, \MM) \le  2{\delta}\tau.
 \eeq 
This and Lemmas~\ref{lem:2.1-12oct} and \ref{lem:2.1new} imply that   for any $\gamma' \ge  \gamma({\delta}),$
\beqs 
- {\delta}\tau +\sqrt{ (2 \sigma^2)\log \left(({\Gamma(H_{\gamma', b})}{\kappa_1})^{-1}\right)}  \leq    \dist(H_{\gamma', b}, \MM) \leq  \sqrt{(2 \sigma^2)\log \left(({\Gamma(H_{\gamma', b})}{\kappa_0})^{-1}\right)}.
\eeqs
and by \eqref{eq on on Gap}, this yields for $\gamma' \ge  \gamma({\delta}),$
\beq \label{eq follows on Gap}
- {\delta}\tau +\sqrt{ (2 \sigma^2)\log \left(({\Gamma(H_{\gamma', b})}{\kappa_1})^{-1}\right)}  \leq    \dist(H_{\gamma', b}, \MM) \leq 
{\delta}\tau +\sqrt{ (2 \sigma^2)\log \left(({\Gamma(H_{\gamma', b})}{\kappa_1})^{-1}\right)} 
\eeq
which proves the claim.
\end{proof}

By Lemma \ref{lem: dist estimate} and formula  \eqref{def of gamma eps}, we have
 \beq
\label{def of gamma eps B}
|  \dist(H_{\gamma(\delta), b}, \MM)-r_\delta| \leq \delta\tau.
 \eeq
where, see \eqref{eq:71},
\beq\label{rdelta def}
r_\delta:=\ell(\Gamma_\delta)= \sqrt{ (2 \sigma^2)\log \left(({  \kappa_1^{-1}\exp\left(- \frac{1}{2} \left(\frac{\sigma}{{\delta} \tau} \log \frac{\kappa_1}{\kappa_0}\right)^2\right)}{\kappa_1})^{-1}\right)}=  \frac{\sigma^2}{{\delta} \tau} \log (\frac{{V}}{(c\sqrt\delta\tau)^d \omega_d }) .
\eeq

As the convex support function $s(b)=\sup_{x\in \MM} b\cdot x$ satisfies
$  \dist(H_{\gamma(\delta), b}, \MM)=\gamma(\delta)-s(b)$,
\eqref{def of gamma eps B} yields that \beq\label{s-def}
|s(b)-(\gamma(\delta)-r_\delta)|<\delta\tau.
\eeq

 Next we want to bound values of $\gamma(\delta) $. 
 \begin{lemma} \label{lem: estimate for gamma} We have
    \beq 
   | \gamma(\delta) - 
 \frac{\sigma^2 }{{\delta}\tau }\log \frac{\kappa_1}{\kappa_0} |=  | \gamma(\delta) - 
 \frac{\sigma^2 }{{\delta}\tau }\log (\frac{{V}}{(c\sqrt\delta\tau)^d \omega_d }) |
 \le 1+ {\delta}\tau \lab{eq:27.1 D}
 \eeq 
 \end{lemma}

 \begin{proof}By \eqref{eq:27.1}, we have
 \beq 
 \sigma^2 \log \frac{\kappa_1}{\kappa_0} = {\delta}\tau \sqrt{ 2 \sigma^2 \log \left((\Gamma_{\delta} \kappa_1)^{-1}\right)}
 = {\delta}\tau \ell(\Gamma_{\delta} ),
 \lab{eq:27.1 B}
 \eeq 
that is,
  \beq 
    \ell(\Gamma(H_{\gamma(\delta) , b}) )=  \ell(\Gamma_{\delta} )=
 \frac{\sigma^2 }{{\delta}\tau }\log \frac{\kappa_1}{\kappa_0}.  \lab{eq:27.1 c}
 \eeq 
Moreover, by \eqref{dist estimate}
 \beq\label{dist estimate 2}
|\ell \left(\Gamma(H_{\gamma(\delta) , b})\right)-\dist(H_{\gamma(\delta) , b}, \MM)| 
\leq {\delta}\tau 
\eeq
 and as $\MM\subset B^D_1(0)$, we have
 \beq\label{dist estimate 3}
 | \gamma(\delta) -\dist(H_{\gamma(\delta) , b}, \MM)|\le 1.
 \eeq
 By combining \eqref{eq:27.1 c}, \eqref{dist estimate 2}, and \eqref{dist estimate 3}, we obtain
the claim.
\end{proof}

 As  we have assumed that $0<\delta\le \eps/16<\frac{\sigma^2 }{16\tau }$, 
 Lemma \ref{lem: estimate for gamma}  implies that $ \gamma(\delta)>2$.

As $\MM\subset B^D_1(0)$, it holds that $$\dist(H_{\gamma_2, b}, \MM)=\dist(H_{\gamma_1, b}, \MM)+\gamma_2-\gamma_1$$ for $\gamma_2>\gamma_1\ge 1$. This and  Lemma \ref{lem: dist estimate} 
imply
\beq\label{dist estimate2}
\bigg|\bigg(\ell(\Gamma(H_{\gamma_2, b}))-\ell (\Gamma(H_{\gamma_1, b}))\bigg)-(\gamma_2-\gamma_1)\bigg|
\leq 2{\delta}\tau 
\eeq
for $\gamma_1,\gamma_2 \ge \gamma({\delta}).$

 \begin{lemma}\label{lem: D0 estimate}
  The derivative of the function $\ell:(0,e/\kappa_1]\to \R$ satisfies
   \beqs
  | \frac {d}{d\Gamma}\ell(\Gamma)|\le
  {\sqrt 2}\sigma
 \frac{{V}}{(c\sqrt\delta\tau)^d \omega_d }\frac{1}{(\sqrt{2\pi}\sigma)^{-1}}\exp\left(\frac{1}{2} \left(\frac{\sigma}{{\delta} \tau} \log ( \frac{{V}}{(c\sqrt\delta\tau)^d \omega_d })\right)^2\right)=:\mathcal D_0
 \eeqs
for $\Gamma_\delta/2\le\Gamma\le e/\kappa_1$. 
 \end{lemma}
 
  \begin{proof}
 The derivative of the function $\ell:(0,e/\kappa_1]\to \R$ is
 \beq
 \frac {d}{d\Gamma}\ell(\Gamma)
 ={\sigma \sqrt{2} } 
 \frac 12\frac 1 {( \log ((\Gamma\kappa_1)^{-1}))^{1/2}} \cdot \frac {1} {(\Gamma\kappa_1)^{-1}}\cdot  \frac {(-1)} {(\Gamma\kappa_1)^{2}}\kappa_1
 =- \frac{\sigma \kappa_1} {\sqrt 2}
\bigg(\frac {((\Gamma\kappa_1)^{-1})^2} {( \log ((\Gamma\kappa_1)^{-1}))}\bigg)^{1/2}.
 \eeq
As $h:s\to  {s^2} /{\log (s)}$ is increasing for $s\ge e$
as its 
 derivative 
 $$\frac {dh}{ds}=\frac {2s} {\log (s)} -\frac {s^2} {\log^2 (s)}\cdot \frac 1s=\frac {s} {\log^2 (s)}(2{\log (s)} -1)$$ is positive 
for $s\ge e^{1/2}$, we see that $| \frac {d}{d\Gamma}\ell(\Gamma)|$ is decreasing for $\Gamma\le e^{-1/2}/\kappa_1$
and increasing for  $\Gamma_0=e/\kappa_1\ge \Gamma\ge e^{-1/2}/\kappa_1$.

As 
$\Gamma_{\delta}  \le e^{-8}\kappa_1^{-1}$, we see that 
$| \frac {d}{d\Gamma}\ell(\Gamma)|$ is decreasing for $\Gamma\le \Gamma_{\delta} /2$

Thus,
for $ e^{-1/2}/\kappa_1\ge \Gamma\ge  \Gamma_{\delta}/2$, 
 \beqs
| \frac {d}{d\Gamma}\ell(\Gamma)|&\le& | \frac {d}{d\Gamma}\ell(\Gamma)\bigg|_{\Gamma=\Gamma_{\delta}/2}|
\eeqs
where
 \beqs
  | \frac {d}{d\Gamma}\ell(\Gamma)\bigg|_{\Gamma=\Gamma_{\delta}}|
  &=&
  {\sigma} 
\frac 1 {( 2\log ((\Gamma\kappa_1)^{-1}))^{1/2}} \cdot   \frac {1} {\Gamma}\bigg|_{\Gamma=\Gamma_{\delta}/2} \\
&\le&
\frac {\sigma} {\sqrt 2}
  \frac {1} {\Gamma}\bigg|_{\Gamma=\Gamma_{\delta}/2} \\
&\le&
  {{\sqrt 2}\sigma} 
 \frac{{V}}{(c\sqrt\delta\tau)^d \omega_d }\frac{1}{  (\sqrt{2\pi}\sigma)^{-1}}\exp\left(\frac{1}{2} \left(\frac{\sigma}{{\delta} \tau} \log ( \frac{{V}}{(c\sqrt\delta\tau)^d \omega_d })\right)^2\right).
 \eeqs

For  $\Gamma_0=e/\kappa_1\ge \Gamma\ge e^{-1/2}/\kappa_1$,
 \beqs
  | \frac {d}{d\Gamma}\ell(\Gamma)&\le&   | \frac {d}{d\Gamma}\ell(\Gamma)\bigg|_{\Gamma=\Gamma_0}
  \\
  &=&
  {\sigma} 
\frac 1 {( 2\log ((\Gamma\kappa_1)^{-1}))^{1/2}} \cdot   \frac {1} {\Gamma}\bigg|_{\Gamma=\Gamma_0} \\
   \\
  &=&
  {\sigma} 
\frac 1 {2^{1/2}} \frac {\kappa_1}e=  
\frac  {\sigma} {2^{1/2}e}\frac{{V}}{(c\sqrt\delta\tau)^d \omega_d }\frac{1}{  (\sqrt{2\pi}\sigma)^{-1}} =:\mathcal D_1.
\eeqs
It holds that $\mathcal D_1\le \mathcal D_0$, and thus 
$\Gamma\in(0,e/\kappa_1]$.
  \end{proof}

Next we consider a moving averaged  of function $\gamma'\to H_{\gamma', b}$, defined by
\beq 
\Gamma^{av}_ {\delta}(H_{\gamma', b}) := \frac 1{\delta}\int\limits_{\gamma'}^{\gamma'+{\delta}} \Gamma(H_{\tilde{\gamma}, b})d\tilde{\gamma} .\lab{eq:nearfinn}
\eeq

\begin{lemma}
\label{lem: Gamma av decreasing}
Function $\gamma\to \Gamma^{av}_ {\delta}(H_{\gamma, b})$, defined for  $\gamma \ge 1$ is strictly decreasing
and satisfies
\beq\label{important eq pre}
\ell(\Gamma(H_{\gamma'+{\delta}, b}))> \ell(\Gamma^{av}_ {\delta}(H_{\gamma', b})) >\ell(\Gamma(H_{\gamma', b})).
\eeq
\end{lemma}

\begin{proof}
As  $\gamma\to \Gamma(H_{\gamma, b})$ is a strictly decreasing function for $\gamma \ge 1$, we see that
the function  ${\gamma}\to \Gamma^{av}_ {\delta}(H_{{\gamma}, b})$ is a strictly decreasing function for
 ${\gamma}\ge 1,$ too.

Moreover, as  $\gamma\to \Gamma(H_{\gamma, b})$ is a strictly decreasing function for $\gamma \ge 1$, 
\beq\label{important eq pre0}
\Gamma(H_{\gamma'+{\delta}, b})< \Gamma^{av}_ {\delta}(H_{\gamma', b}) <\Gamma(H_{\gamma', b})
\eeq
and, as $\ell$ is a strictly decreasing function, formula \eqref{important eq pre} follows.
\end{proof}

By  Lemma \ref{lem: estimate for gamma},  $ \gamma(\delta)>2$. As $\delta<1$, we see that $\Gamma^{av}_ {\delta}(H_{\gamma, b})|_{\gamma=1}
>\Gamma(H_{\gamma, b})|_{\gamma=2}>\Gamma_\delta$. Moreover,
 $ \Gamma^{av}_ {\delta}(H_{\gamma, b})\to 0$ as $\gamma\to \infty.$ These and
 Lemma \ref{lem: Gamma av decreasing} imply  there is a unique $\gamma^{av}(\delta)\in \R$  such that 
 \beq\label{def: tilde gamma eps, delta1}
 \Gamma^{av}_ {\delta}(H_{\gamma^{av}(\delta),b})=\Gamma_\delta.
 \eeq

Our next aim is to estimate values of $\Gamma^{av}_ {\delta}(H_{\gamma', b})$ using the sampled data points.
To this end, we appeal to results involving  Vapnik-Chervonenkis dimension theory.

\begin{definition}[Definition  8.3.1, \cite{vershynin}] 
Consider a class $\mathcal C$ of Boolean functions on { a set} ${ S}.$ We say that a subset $\Gamma \subseteq { S}$ is {\it shattered} by $\mathcal C$ if any function $g: \Gamma \ra \{0, 1\}$ can be obtained by restricting some function $f \in \mathcal C$ to $\Gamma.$ The VC dimension of $\mathcal C$ denoted $vc(\CC)$, is the largest cardinality of a subset $\Gamma \subseteq { S}$ shattered by $\CC.$
\end{definition}
Consider a family $\mathcal {C}$ of subsets of the sample space ${ S}$. 
{We say that the family $\mathcal {C}$ picks out} a certain subset $W$ of the finite set $S=\{x_{1},\ldots ,x_{s}\}\subset { S} $ if $W=S\cap f$ for some $f \in {\mathcal  {C}}.$  Moreover, 
we say that ${\mathcal {C}}$  shatters $S$ if it picks out each of its $2^s$ subsets. 
 Below,
the set $\{{\bf 1}_f\ |\ f\in \mathcal {C}\}$ of the indicator functions of the sets $f\in \mathcal {C}$
picks our a set $W$ if and only if  a family $\mathcal {C}$ of subsets of the set ${ S}$ picks out a set $W$.
Similarly, the VC-dimension of the a family $\mathcal {C}$ of sets is the VC-dimension of the
 set $\{{\bf 1}_f\ |\ f\in \mathcal {C}\}$ of the indicator functions.

For $(z_{1},\ldots ,z_{s})\in (\R^D)^s$, 
let $\Delta _{D}({\mathcal  {C}},z_{1},\ldots ,z_{s}) $ be the number of different subsets  $W\subset 
\{z_{1},\ldots ,z_{s}\}$ that are picked out by  ${\mathcal {C}}$, that is,
$\Delta _{D}({\mathcal  {C}},z_{1},\ldots ,z_{s}) $ is the number of different sets
in the family $\{f\cap \{z_{1},\ldots ,z_{s}\}:\ f\in {\mathcal {C}}\}$.
Theorem 13.3 of  \cite{gyorfi} gives the following:

\begin{lemma}[Sauer-Shelah  lemma] Let  $(z_{1},\ldots ,z_{s})\in (\R^D)^s$ and let ${\mathcal  {C}}$
be a family of subsets of the set $\{z_{1},\ldots ,z_{s}\}\subset \R^D$.
Then  the number $\Delta _{D}({\mathcal  {C}},z_{1},\ldots ,z_{s}) $ satisfies:
\beq
 \Delta_D({\mathcal  {C}},z_{1},\ldots ,z_{s})\leq \sum _{{j=0}}^{{vc({\mathcal  {C}})}}{s \choose j}\leq \left({\frac  {s\cdot e}{vc({\mathcal  {C}})}}\right)^{{vc({\mathcal  {C}})}}
\eeq
\end{lemma}

The following Theorem of Vapnik and Chervonenkis is paraphrased from Theorem 12.5 of  \cite{gyorfi}.

\begin{theorem}[Vapnik and Chervonenkis]  \label{thm: VC} 
{  Let $\CC$ be a class of Boolean functions on a set $S\subset \R^D$ with finite VC dimension $vc(\CC) \geq 1$
and $\mu$ be a probability measure on $S$.
 Let $X, X_1, \dots,   X_N$ be independent, indentically distributed $S$-valued random variables having
 the distribution $\mu$. Then
$$
\p\left[\sup\limits_{f \in \CC}\left|\frac{1}{N}\sum_{i=1}^N f(X_i) - \E f(X) \right|  > \mathbf{a}\right] \leq 8\max _{(z_{1},\ldots ,z_{N})\in (\R^D)^N}\Delta _{D}({\mathcal  {C}},z_{1},\ldots ,z_{N})  \exp\left(- \frac{N \mathbf{a}^2}{32}\right).
$$
}
\end{theorem}

\begin{theorem}[Theorem 3.1 in  \cite{dudley}]\label{thm: Dudley}
Let $\FF$ consist of the set of indicators of halfspaces in $\R^D.$ Then $vc(\FF) = D+1.$
\end{theorem}

Recall that by (\ref{eq:71}) that 
\beqs  \Gamma_\delta  = \frac{(c\sqrt\delta\tau)^d \omega_d }{{V}}   (\sqrt{2\pi}\sigma)^{-1}\exp\left(- \frac{1}{2} \left(\frac{\sigma}{\delta \tau} \log(\frac{V}{(c\sqrt{\delta}\tau)^d \omega_d }) \right)^2\right). \eeqs

Let 
$f_\gamma:\R^D\to \{0,1\}$ 
{be the indicator functions of the halfspaces $\{ y\in \R^D:\ \langle b,y \rangle> \gamma\}$,
that is,
$$
f_\gamma(y)={\bf 1}_{\{ \langle b,y \rangle> \gamma\}}(y).
$$
Moverover, let
 $\CC$ be the family of  the indicator functions
of halfspaces,  the measure $\mu$ have the probability density function $\rho$ in $\R^D$,
 \beq 
 \mathbf{a} 
 &:=&\frac 
{{\delta^2}}6  \mathcal D_0^{-1}\\
&=&
\frac 
{\delta^2} 6 \cdot
 \bigg( {2\sigma} 
 \frac{{V}}{(c\sqrt\delta\tau)^d \omega_d }\frac{1}{  (\sqrt{2\pi}\sigma)^{-1}}\exp\left(\frac{1}{2} \left(\frac{\sigma}{{\delta} \tau} \log ( \frac{{V}}{(c\sqrt\delta\tau)^d \omega_d })\right)^2\right)\bigg)^{-1}
\\
&=&\frac {\delta^2}{12\sqrt{2\pi}\sigma^2} \cdot
 \frac{(c\sqrt\delta\tau)^d \omega_d }{V} \exp\left(-\frac{1}{2} \left(\frac{\sigma}{{\delta} \tau} \log ( \frac{{V}}{(c\sqrt\delta\tau)^d \omega_d })\right)^2\right)
  \eeq 
and the number $N$ of samples satisfies} 
\beq \label{eq: N condition}
N\mathbf{a}^2  & \geq & 32\log \bigg(8\eta^{-1}  ({\frac  {N\cdot e}{D+1})^{D+1}}\bigg).
\eeq 
The value of $\mathbf{a}$ is chosen so that we have $ \mathcal D_0 \frac {\mathbf{a}}{\delta}\leq \frac \delta 6.$

In the above setting, Theorems \ref{thm: VC} and \ref{thm: Dudley} imply that 
 that 
\beq\label{eq final}
\p\left[\sup\limits_{\gamma\in \R}\left|\frac{1}{N}\sum_{i=1}^N f_\gamma(X_i) - \E f_\gamma(X) \right|  > \mathbf{a}\right] \leq \eta.
\eeq
where
\beq\label{concentration estimate1}
\E f_\gamma(X)=
\int\limits_{\{y\in \R^D:\ \langle b,y \rangle< \gamma\}}\rho(y)dy
=
\int\limits_{\tilde{\gamma} \in (-\infty,\gamma)} \Gamma(H_{\tilde{\gamma}, b})d\tilde{\gamma}.
\eeq
We consider the random variables
$$
F_{\gamma,N}=\frac{1}{N}\sum_{i=1}^N f_\gamma(X_i)
$$
and, motivated by the formula 
\beq
\Gamma^{av}_ {\delta}(H_{j{\delta}})=\frac 1{\delta}(\E f_{j\delta}(X)-\E f_{j\delta+\delta}(X)),
\eeq
 we define the random variables
\beq
\Gamma^{est}_N(H_{j\delta, b})=\frac 1{\delta}(F_{j\delta,N}-F_{j\delta+\delta,N}).
\eeq
By \eqref{concentration estimate1}, 
with probability $1-\eta$, we have that for all $j\in \mathbb Z$, we have 
\beq\label{error est}
|\Gamma^{est}_N(H_{j\delta, b})-\Gamma^{av}_ {\delta}(H_{j\delta, b})|\le 2\delta^{-1}\mathbf{a}
\eeq

Recall that $\gamma^{av}(\delta)$ is defined in \eqref{def: tilde gamma eps, delta1},
 \beq\label{def: tilde gamma eps, delta1 cp}
 \Gamma^{av}_ {\delta}(H_{\gamma^{av}(\delta),b})=\Gamma_\delta.
 \eeq
 Next, we approximate $\gamma^{av}(\delta)$ by an element of the discrete set $\delta \mathbb Z$.
Let $j^{av}$ be such an integer that 
\beq\label{eq j ordering}
(j^{av}-1)\delta\le \gamma^{av}(\delta) < j^{av}\delta.
\eeq
As $\gamma'\to  \Gamma^{av}_ {\delta}(H_{\gamma', b})$ is a strictly decreasing function for $\gamma'\ge 1$,
we see using \eqref{eq j ordering} that
\beq\label{important eq combined5a}
 \Gamma^{av}_ {\delta}(H_{j^{av}\delta, b}) <\Gamma_\delta=\Gamma^{av}_ {\delta}(H_{\gamma^{av}(\delta),b})\le  \Gamma^{av}_ {\delta}(H_{(j^{av}-1)\delta, b}). 
\eeq
and that $j^{av}$ is the smallest integer $j\in \Z$ such that $\delta j\ge 1+\delta$ and
\beq\label{j-av property}
\Gamma^{av}_ {\delta}(H_{j\delta, b})\le \Gamma_\delta.
\eeq

Motivated by the fact that $j^{av}$ is the smallest integer satisfying \eqref{j-av property}, 
we define  $j_*$ to be the smallest such integer $j\in \Z$ such that $\delta j\ge 1+\delta$ and
\beq\label{Gamma estimated}
\Gamma^{est}_N(H_{j\delta, b})\le \Gamma_\delta
\eeq
or infinity, if no such integer exists.

\begin{lemma}\label{lem: s function}
When $N$ satisfies \eqref{eq: N condition}, the smallest integer $j_*\geq \delta^{-1}+1$
for which \eqref{Gamma estimated} holds satisfies 
\beq\label{important eq combined11}
|s(b)-(j_*\delta-r_\delta)|<\eps,
\eeq
with probability larger or equal to $1-\eta$,
where $s(b)=\sup_{x\in \MM} b\cdot x$ is the convex support function of $\MM$ and 
$r_\delta$ is given in \eqref{rdelta def}.
\end{lemma}
\begin{proof}
When  $j_*$ is finite,
$$
\Gamma^{est}_N(H_{j_*\delta, b})\le \Gamma_\delta
\quad\hbox{and}\quad
\Gamma^{est}_N(H_{(j_*-1)\delta, b})> \Gamma_\delta,
$$
so that  by \eqref{error est}
$$
\Gamma^{av}_ {\delta}(H_{j_*\delta, b})\le \Gamma_\delta+ 2\delta^{-1}\mathbf{a}
\quad\hbox{and}\quad
\Gamma^{av}_ {\delta}(H_{(j_*-1)\delta, b})> \Gamma_\delta-2\delta^{-1}\mathbf{a}
$$
Then,
\beq\label{important eq combined5}
 \Gamma^{av}_ {\delta}(H_{j_*\delta, b})- 2\delta^{-1}\mathbf{a} \le\Gamma_\delta=\Gamma^{av}_ {\delta}(H_{\gamma^{av}(\delta),b})< \Gamma^{av}_ {\delta}(H_{(j_*-1)\delta, b})+ 2\delta^{-1}\mathbf{a}
\eeq
By applying \eqref{important eq pre0} with $\gamma'$ having the values $j_*\delta$ and $(j_*-1)\delta$, we obtain
\beq\label{important eq pre0 c}
& &\Gamma(H_{(j_*+1)\delta, b})<\Gamma^{av}_ {\delta}(H_{j_*\delta', b}),\\
\nonumber
& &\Gamma^{av}_ {\delta}(H_{(j_*-1)\delta, b})<\Gamma(H_{(j_*-1)\delta, b}),
\eeq
and thus by \eqref{important eq combined5},
\beq\label{important eq combined6}
 \Gamma(H_{(j_*+1)\delta, b})- 2\delta^{-1}\mathbf{a} <\Gamma_\delta
 < \Gamma(H_{(j_*-1)\delta, b})+ 2\delta^{-1}\mathbf{a}.
\eeq
As by Lemma \ref{lem: D0 estimate}, $\ell(\Gamma)$ is a strictly decreasing function which derivative is bounded by $\mathcal D_0$
on the interval  $ 2 \Gamma_{\delta}\ge \Gamma\ge  \Gamma_{\delta}/2$,  we have
\beq\label{important eq combined8}
 \ell(\Gamma(H_{(j_*+1)\delta, b}))+ 2\mathcal D_0\delta^{-1}\mathbf{a} >\ell(\Gamma_\delta)
 =\ell( \Gamma(H_{\gamma({\delta}), b}))
 > \ell(\Gamma(H_{(j_*-1)\delta, b}))- 2\mathcal D_0\delta^{-1}\mathbf{a}.
\eeq
This implies
\beq\label{important eq combined8b}\\
\nonumber
 \hspace{-5mm}\bigg( \ell(\Gamma(H_{(j_*+1)\delta, b}))-\ell( \Gamma(H_{\gamma({\delta}), b}))\bigg)+ 2\mathcal D_0\delta^{-1}\mathbf{a} >0
 >\bigg( \ell(\Gamma(H_{(j_*-1)\delta, b}))-\ell( \Gamma(H_{\gamma({\delta}), b}))\bigg)- 2\mathcal D_0\delta^{-1}\mathbf{a}
 \hspace{-15mm}
\eeq

By combining inequalities \eqref{dist estimate2},
and \eqref{important eq combined8b}, we see that 
\beq\label{important eq combined9}
\bigg((j_*+1)\delta-\gamma(\delta) \bigg)+ 2\mathcal D_0\delta^{-1}\mathbf{a} +2 \delta\tau>0> \bigg((j_*-1)\delta
-\gamma(\delta) \bigg)
-2\mathcal D_0\delta^{-1}\mathbf{a}- 2\delta\tau
\eeq
which implies
\beq\label{important eq combined9b}
|j_*\delta-\gamma(\delta) |\leq \delta+ 2\mathcal D_0\delta^{-1}\mathbf{a} +2 \delta\tau.
\eeq
As $\delta\tau+(\delta+ 2\mathcal D_0\delta^{-1}\mathbf{a} +2 \delta\tau)<\eps$, formulas \eqref{eq final} and \eqref{important eq combined9b} imply
the claim.
\end{proof}

Next, we give an estimate for the number $N$ of samples in terms of $\eps$, $\tau$, $\sigma$, and $D$.
Recall that the estimate \eqref{eq final}  holds when the number $N$ satisfies
\beq \nonumber
N\mathbf{a}^2  & \geq & 32\log \bigg(8\eta^{-1}  ({\frac  {N\cdot e}{D+1})^{D+1}}\bigg)\\
& = &  32(\log 8+ \log \eta^{-1} +(D+1) \log({N\cdot e})-(D+1) \log({D+1})) .\label{N rec}
\eeq
To analyze the above inequality, we observe that when
   \beq   \label{ineq for N}
N\ge  \mathbf{a}^{-2}\log(\eta^{-1})\log^2  (\mathbf{a}^{-1})
 \eeq
holds, we can write 
  \beqs
N= \mathbf{a}^{-2}\log(\eta^{-1})\log^2  (\mathbf{a}^{-1})\cdot m,\quad\hbox{where }m\ge 1
 \eeqs
 and then the right hand side of  \eqref{N rec} satisfies
 \beq
 & & 32(\log 8+ \log \eta^{-1} +(D+1) \log({N\cdot e})-(D+1) \log({D+1}))\nonumber\\
 &\le
  &  32 \mathbf{a}^{-2}\bigg(\log 8+ \log (\eta^{-1}) +(D+1) (1+\log({N}))\bigg)\lab{eq: N co}\\
&=&  32 \mathbf{a}^{-2}\bigg(\log 8+ \log (\eta^{-1}) +(D+1) (1+\log( \mathbf{a}^{-2}\log(\eta^{-1})\log^2  (\mathbf{a}^{-1})m))\bigg)\nonumber\\
&\le &  32 \mathbf{a}^{-2}\bigg(\log 8+2D+ \log (\eta^{-1}) +2D \log( \mathbf{a}^{-2})\nonumber\\
& &+2D \log(\log(\eta^{-1}))
+2D \log(\log^2  (\mathbf{a}^{-1}))+2D \log(m))\bigg)\nonumber\\
&\le &  32 \mathbf{a}^{-2}\bigg(\log 8+2D+2D \log(2)+3D \log (\eta^{-1}) +8D \log( \mathbf{a}^{-1})
+2D m)\bigg)\nonumber
\\ &\le &  N\nonumber
 \eeq
 assuming that 
 \beq\label{m estimate 1}
 &&32 (\log 8+2D+2D \log(2))\leq \frac 14\frac1{\log(\eta^{-1})\log^2  (\mathbf{a}^{-1})\cdot m},\\
 & &32\cdot 3D \leq \frac 14\frac1{\log^2  (\mathbf{a}^{-1})\cdot m},\\
 & &  32 \cdot 8D \leq \frac 14\frac 1{\log(\eta^{-1})\log (\mathbf{a}^{-1})\cdot m},\\
 & &  32 \cdot 2D \leq \frac 14\frac1{\log(\eta^{-1})\log^2  (\mathbf{a}^{-1})}\label{m estimate 4}.
 \eeq
 We see that there is $\eps_0>0$, depending on  $\eps$, $\tau$, $\sigma$, and $D$, such that if $0<\eps\le \eps_0$ then 
 $\mathbf{a}$ is so  small that inequalities \eqref{m estimate 1}-\eqref{m estimate 4} are valid, and hence, inequality
 \eqref{N rec} is valid. 
Summarizing, the
inequality \eqref{ineq for N} and thus \eqref{eq final} are valid when $0<\eps<\eps_0$ and
    \beq   \label{ineq for N co}
N\ge\bigg( \frac 
{3\cdot 2^{10}\sqrt{2\pi}\sigma^2}{\delta^2} \cdot
 \frac{V}{(c\sqrt{{\eps}}\tau/4)^d \omega_d }\bigg)^3 \exp\left( \left(\frac{\sigma}{{\eps} \tau/4} \log ( \frac{{V}}{(c\sqrt{\eps}\tau/4)^d \omega_d })\right)^2\right)\log(\eta^{-1})
 \eeq
 This and Lemma \ref{lem: s function}
yield the Proposition \ref{lem:2.1}

\end{proof}

 \section{The measure of points deep in the relative interior of an outer normal cone.}

\begin{definition}
Given a point $x \in \MM$, we define $\mS(x)$ to be the intersection of $B_1(x)$ with the relative interior of the outer normal cone $N_K(x).$
\end{definition}

Denote by $\Upsilon\MM$ the normal bundle of $\MM^d\subset\R^D$.
By definition, $\Upsilon\MM$ is the set of pairs
$$
 \Upsilon\MM = \{(p,v) : p\in\MM, v\in T_p^\perp\MM \} .
$$
It is an $D$-dimensional $C^1$ submanifold of $\R^D\times\R^D$.
Let $\pi\colon\Upsilon\MM\to\MM$ and
$I\colon\Upsilon\MM\to\R^D$ be the maps defined by
$$
 \pi(p,v) = p
$$
and
$$
 I(p,v) = v .
$$

Our goal is to estimate the volumes of the $I$-images in $\R^D$
of various subsets of $\Upsilon\MM$.

Consider a point $(p,v)\in \Upsilon\MM$ and the tangent space
$T_{(p,v)}(\Upsilon\MM)$ of $\Upsilon\MM$ at this point.
This tangent space is a linear subspace of $\R^D\times\R^D$
and we write its elements as $(\xi,\eta)$ where $\xi,\eta\in\R^D$.
Note that $\xi\in T_p\MM$ since $\Upsilon\MM\subset\MM\times\R^D$.
The following lemma is a standard property of derivatives of normal
vector fields rewritten with our notation.

\begin{lemma} \label{l:derivative of normal}
If $(\xi,\eta)\in T_{(p,v)}(\Upsilon\MM)$ then for every $\xi_1\in T_p\MM$
one has 
\beq \label{e:derivative of normal 0}
 \langle \eta,\xi_1 \rangle = -\langle \II_p(\xi,\xi_1), v \rangle
\eeq
where $\II_p$ is the second fundamental form of $\MM$ at $p$.
\end{lemma}

\begin{proof}
Pick a coordinate chart on $\MM$ near $p$ and extend $\xi$ and $\xi_1$
to local vector fields that have constant coordinates in the chart.
Since the vector $(\xi,\eta)$ is tangent to $\Upsilon\MM$ at $(p,v)$,
there exists a $C^1$ curve $t\mapsto (\gamma(t),\nu(t))$ in $\Upsilon\MM$,
where $t$ ranges over some interval containing~0,
such that $\gamma(0)=p$, $\nu(0)=v$, $\gamma'(0)=\xi$ and $\nu'(0)=\eta$.
Since $ (\gamma(t),\nu(t)) \in \Upsilon\MM$ and $\xi_1$ is a tangent
vector field, we have
$$
 \langle \nu(t), \xi_1(\gamma(t)) \rangle = 0
$$
for all~$t$. Differentiation of this identity at $t=0$ yields that
$$
 \langle \nu'(0), \xi_1 \rangle + \langle \nu(0), \xi_1(\gamma(t))'_{t=0} \rangle
 = 0 .
$$
Since $\nu(0)=v$ and $\nu'(0)=\eta$, this can be rewritten as
\beq \label{e:derivative of normal 1}
 \langle \eta, \xi_1 \rangle = -\langle v, \xi_1(\gamma(t))'_{t=0}\rangle.
\eeq
The term $\xi_1(\gamma(t))'_{t=0}$ is the second differential of the
local parametrization of $\MM$ evaluated at the pair of directions $(\xi,\xi_1)$.
Hence 
\beq \label{e:derivative of normal 2}
\Pi_{T_p^\perp\MM}(\xi_1(\gamma(t))'_{t=0}) = \II_p(\xi,\xi_1) 
\eeq
by the definition of the second fundamental form.
Since $v\in T_p^\perp\MM$,
\eqref{e:derivative of normal 1} and \eqref{e:derivative of normal 2}
imply \eqref{e:derivative of normal 0}.
\end{proof}

We decompose $T_{(p,v)}(\Upsilon\MM)$ into the direct sum
$$
 T_{(p,v)}(\Upsilon\MM)
 = V_{p,v} \oplus H_{p,v}
$$
of the ``vertical'' subspace $V_{p,v}$ 
and the ``horizontal'' subspace $H_{p,v}$ defined as follows.
The vertical subspace is given by
$$
 V_{p,v} := \{0\} \times T_p^\perp\MM ,
$$
it is essentially the tangent space to a fiber of the normal bundle.
The horizontal subspace $H_{p,v}$ is defined as the orthogonal complement
of $V_{p,v}$ in $T_{(p,v)}(\Upsilon\MM)$ with respect to
the scalar product inherited from $\R^D\times\R^D$.

A vector $(\xi,\eta)\in\R^D\times\R^D$ belongs to $H_{p,v}$
if and only if $(\xi,\eta)\in T_{(p,v)}(\Upsilon\MM)$
and $\eta\in (T_p^\perp\MM)^\perp=T_p\MM$.
Recall that $\xi$ also belongs to  $T_p\MM$,
hence $H_{p,v}$ is a subset of $T_p\MM\times T_p\MM$ and moreover
\beq\label{e:Hpv1}
 H_{p,v} = \{(\xi,\eta)\in T_p\MM\times T_p\MM \mid
 (\xi,\eta)\in T_{(p,v)}(\Upsilon\MM) \} .
\eeq
Clearly $\dim V_{p,v}=D-d$, $\dim H_{p,v}=d$ and
$V_{p,v}=\ker d_{(p,v)}\pi$ 
where
$$
d_{(p,v)}\pi\colon T_{(p,v)}(\Upsilon\MM) \to T_p\MM
$$
is the differential of $\pi$ at $(p,v)$.
Therefore $d_{(p,v)}\pi$ maps $H_{p,v}$ to $T_p\MM$ bijectively.
Since $d_{(p,v)}\pi(\xi,\eta)=\xi$, it follows that
for every $\xi\in T_p\MM$ there exists a unique $\eta\in T_p\MM$
such that $(\xi,\eta)\in H_{p,v}$.
We denote this unique $\eta$ by $L_{p,v}(\xi)$,
clearly this defines a linear map 

\beq L_{p,v} \colon T_p\MM \to T_p\MM .\lab{eq:80A}\eeq 

Thus
\beq\label{e:Hpv2}
 H_{p,v} = \{ (\xi,\eta) \in T_p\MM\times T_p\MM \mid \eta = L_{p,v}(\xi) \}
\eeq
and the result of Lemma \ref{l:derivative of normal} takes the form
\beq \label{e:Lpv and II}
 \langle L_{p,v}(\xi), \xi_1 \rangle = -\langle \II_p(\xi,\xi_1), v \rangle
\eeq
for all $\xi,\xi_1\in T_p\MM$ and $v\in T_p^\perp\MM$.

{\bf Remark:} Now one can see that the map $L_{p,v}$ is nothing but
the shape operator of $\MM$ with respect to the normal vector $v$.

\begin{lemma} \label{l:area formula}
For every Borel measurable set $A\subset \Upsilon\MM$,
\beq\label{e:area formula 0}
 \int_{\MM} \int_{A\cap T_p^\perp\MM} |\det(L_{p,v})| \,dv \,dp
 = \int_{\R^D} \#(A\cap I^{-1}(y))\, dy
\eeq
where $dv$, $dp$, $dy$ denote the $(D-d)$-dimensional Euclidean volume element
on $T_p^\perp\MM$, the $d$-dimensional Riemannian volume element on $\MM$,
and the $D$-dimensional Euclidean volume element in $\R^D$, resp.,
$\det L_{p,v}$ is the determinant of the linear operator $L_{p,v}$
on $T_p\MM$ defined in (\ref{eq:80A}),
and $\#$ denotes the cardinality of a set.
\end{lemma}

\begin{proof}
This is essentially the area formula for the map $I$ restricted to $A$.
The right-hand side is the volume of $I(A)$ counted with multiplicity.
We equip $\Upsilon\MM$ with smooth volume given by
$dv\,dp$ (as in the left-hand side of \eqref{e:area formula 0})
and use the area formula
$$
 \int_{\MM} \int_{A\cap T_p^\perp\MM} J(d_{(p,v)}I) \,dv \,dp
 = \int_{\R^D} \#(A\cap I^{-1}(y))\, dy
$$
where $J(d_{(p,v)}I)$ is the volume expansion ratio
of the differential $d_{(p,v)}I$.
This differential has the form
$$
 d_{(p,v)}I(\xi,\eta) = \eta, \qquad (\xi,\eta)\in T_{(p,v)}(\Upsilon\MM),
$$
and it respects the orthogonal decompositions
$T_{(p,v)}(\Upsilon\MM) = V_{p,v} \oplus H_{p,v}$ of the domain
and $\R^n=T_p^\perp\MM\oplus T_p\MM$ of the target.
The volume element $dv$ on the first term of the decomposition 
is preserved, and the volume element $dp$ on the second term
is multiplied by $|\det(L_{p,v})|$ due to~\eqref{e:Hpv2}.
This implies \eqref{e:area formula 0}.
\end{proof}

\begin{lemma}\label{l:jacobian upper bound}
Assume that $\reach(\MM)\ge\tau$.
Then, for all $p\in\MM$ and $v\in T_p^\perp\MM$,
$$
 \|L_{p,v}\| \le \tau^{-1} |v|
$$
and therefore
$$
 |\det(L_{p,v})| \le \tau^{-d} |v|^d .
$$
\end{lemma}

\begin{proof}
The inequality $\reach(\MM)\ge\tau$ implies that $\|S_p\|\le \tau^{-1}$.
Hence, by \eqref{e:Lpv and II}, we have
$$
 |\langle L_{p,v}(\xi), \xi_1 \rangle| \le \tau^{-1}|\xi||\xi_1|
$$
for all $\xi,\xi_1\in T_p\MM$.
Since $L_{p,v}$ is an operator on $T_p\MM$,
it follows that $|L_{p,v}(\xi)|\le \tau^{-1}|\xi|$ for all $\xi\in T_p\MM$,
hence the result.
\end{proof}

\begin{corollary}\label{cor:cones volume lower bound}
Recall that $\mS_K(p)$ denotes the intersection of $N_K(p)$ with the unit ball $B_1(0)$,
where $K=\conv(\MM)$ and $\reach(\MM)\ge\tau$.
Then
$$
 \int_{\MM} \vol_{D-d}(\mS_K(p)) \,dp \ge \tau^d \omega_D
$$
where $dp$ in the Riemannian volume element on $\MM$
and $\omega_D$ is the volume of the unit ball in~$\R^D$.
\end{corollary}

\begin{proof}
We apply Lemma \ref{l:area formula} to $A=\bigcup_{p\in\MM}\mS_K(p)$.
Since every vector of $\R^n$ belongs to the outer normal cone $N_K(p)$
for some $p\in\MM$, the image $I(A)$ contains the unit ball of~$\R^n$.
Hence the right-hand part of \eqref{e:area formula 0}
is bounded below by $\omega_D$.
By Lemma~\ref{l:jacobian upper bound}, the left-hand side of \eqref{e:area formula 0}
is bounded above by $\tau^{-d}\int_{\MM} \vol_{D-d}(\mS_K(p)) \,dp$,
hence the result.
\end{proof}

\begin{lemma}\label{l:jacobian lower bound}
Assume that $\MM$ is $R$-exposed.
Let $p\in\MM$ and $\nu_p\in T_p^\perp\MM$
be a unit vector constructed in Lemma \ref{l:nu_p}.
Let $v\in N_K(p)$ and $\ep>0$ be such that
$v+\ep\nu_p\in N_K(p)$.
Then
$$
 |L_{p,v}(\xi)| \ge \frac {\ep}{R} |\xi|
$$
for all $\xi\in T_p\MM$, and therefore
$$
 |\det(L_{p,v})| \ge \left(\frac{\ep}{R}\right)^d .
$$
\end{lemma}

\begin{proof}
Let $\gamma$ be a $\C^2$ curve in $\MM$ such that $\gamma(0)=p$
and $\gamma'(0)=\xi$.
Let $v\in N_K(p)$ satisfy the assumption of the lemma.
Then, since $v+\ep\nu_p\in N_K(p)$,
$$
 \langle\gamma(t)-p,v+\ep\nu_p\rangle \le 0
$$
for all~$t$.
On the other hand, by Lemma \ref{l:nu_p} we have
$$
 \langle\gamma(t)-p,\nu_p\rangle \ge \frac1{2R}|\gamma(t)-p|^2 ,
$$
hence
$$
 \langle\gamma(t)-p,v\rangle \le -\frac\ep{2R}|\gamma(t)-p|^2 ,
$$
for all $t$, with equality at $t=0$.
Taking the second derivative at $t=0$ we obtain that
$$
 \langle\gamma''(0), v\rangle \le - \frac\ep{R} |\gamma'(0)|^2
 = -\frac {\ep}R |\xi|^2.
$$
Taking into account \eqref{e:Lpv and II} and the identity
 $\II_p(\xi,\xi) = \Pi_{T_p^\perp\MM}(\gamma''(0))$,
we conclude that
$$
 \langle L_{p,v}(\xi),\xi \rangle
 = -\langle \II_p(\xi,\xi), v\rangle
 \ge \frac \ep R |\xi|^2
$$
and the lemma follows.
\end{proof}

\begin{lemma}\lab{lem:7.8}
Assume that $\reach(\MM)\ge\tau$ and $\MM$ is $R$-exposed.
Let $\ep>0$ and let $G\subset\R^D$ be a Borel measurable set
such that every $v\in G$ belongs to the relative interior
of the outer normal cone $N_K(p)$ for some $p\in\MM$
and moreover $v + \left(B_\eps(0)\cap T_p^\perp\MM\right) \subseteq N_K(p).$ 
Then
$$
 \frac{\vol_D(G)}{\omega_D} 
 \ge \left( \frac{\ep\tau}{R} \right)^d
 \inf_{p\in\MM} \frac{\vol_{D-d}(\mS_K(p)\cap G)}{\vol_{D-d}(\mS_K(p))} .
$$
\end{lemma}

\begin{proof}
Since  $v + \left(B_\eps(0)\cap T_p^\perp\MM\right) \subseteq N_K(p),$ $v$ 
satisfies the assumption from
Lemma~\ref{l:jacobian lower bound} (for some choice of $\nu_p$).
We apply Lemma \ref{l:area formula} to $A=I^{-1}(G)$.
Since $I$ is injective on this set, the right-hand side
of \eqref{e:area formula 0} equals $\vol_D(G)$.
Hence
$$
 \vol_D(G) =  \int_{\MM} \int_{A\cap T_p^\perp\MM} |\det(L_{p,v})| \,dv \,dp
 \ge \left(\frac{\ep}{R}\right)^d \int_{\MM} \vol_{D-d}(\mS_K(p)\cap G) \,dp
$$
where the inequality follows from Lemma \ref{l:jacobian lower bound}.
This and Corollary \ref{cor:cones volume lower bound} imply that
$$
 \frac{\vol_D(G)}{\omega_D} \ge \left(\frac{\ep\tau}{R}\right)^d
 \cdot \frac{\int_{\MM} \vol_{D-d}(\mS_K(p)\cap G) \,dp}
            {\int_{\MM} \vol_{D-d}(\mS_K(p)) \,dp} .
$$
The lemma follows trivially.
\end{proof}

\subsection{Thickness of outer normal cones}

Let $K$ be the convex hull of $\MM$ and $N_K(p)$
the outer normal cone at $p\in M$.
It is easy to see that $N_K(p)\subset T_p^\perp\MM$.
The positive reach and $R$-exposedness
imply thickness of outer normal cones:

\begin{lemma}\label{l:thick outer cone}
Let $\MM\subset\R^n$ be an $R$-exposed manifold with $\reach(\MM)\ge\tau$.
Let $K=conv(\MM)$.
Then for every $p\in\MM$ the outer normal cone $N_K(p)$ contains an
$(n-d)$-dimensional ball of radius $\frac{\tau}{R}$
centered at a unit vector. Consequently, $\mS_K(p)$ contains a $(n-d)$-dimensional ball of radius $\frac{\tau}{\tau + R}.$
\end{lemma}

\begin{proof}
Define $v=-\nu_p$ where $\nu_p$ is a vector from Lemma \ref{l:nu_p}.
Then, by Lemma \ref{l:nu_p},
\beq\label{e:thick outer cone aux1}
 \langle v, q-p \rangle = \langle -\nu_p,q-p\rangle \le -\frac1{2R} |q-p|^2 .
\eeq

Let $B$ be the ball of radius $\frac{\tau}{R}$ in $T_p^\perp\MM$
centered at~$v$. Our goal is to prove that $B\subset N_K(p)$.
Pick $v'\in B$, then $v'=v+w$ for some  $w\in T_p^\perp\MM$
such that $|w|\le\frac{\tau}{R}$.
For every $q\in M$,
\beq\label{e:thick outer cone aux2}
 \langle w, q-p\rangle = \langle w, \Pi_{T_p^\perp\MM}(q-p) \rangle
 \le |w|\cdot |\Pi_{T_p^\perp\MM}(q-p)| \le \frac{|w|}{2\tau}|q-p|^2
 \le \frac1{2R} |q-p|^2
\eeq
where the first equality follows from the fact that $w\in T_p^\perp\MM$,
the second inequality follows from Corollary \ref{cor:reach condition},
and the last one from the bound $|w|\le\frac{\tau}{R}$.
Summing \eqref{e:thick outer cone aux1} and \eqref{e:thick outer cone aux2}
we obtain that
$$
 \langle v',q-p\rangle = \langle v+w,q-p\rangle \le 0 ,
$$
hence $v'\in N_K(p)$. Thus $B\subset N_K(p)$ and the lemma follows. By scaling, $\mS_K(p)$ contains a $(n-d)$-dimensional ball of radius $\frac{\tau}{\tau + R}.$

\end{proof}
\section{Stability of the fiber map $\pi$}

\begin{lemma}\lab{lem:8.1}

Let $v_p \in \mS_K(p)$ be a point in the outer normal cone at $p \in \MM$, such that a $D-d$ dimensional disc $D_0(x)$ of radius $r_0$ centered at $v_p$ is contained in $\mS_K(p).$   Suppose $|v' - v_p| < \de.$ Then defining $L$ as in Theorem~\ref{thm:stable},  for some  $p'$ such that $\|p -  p'\| < \left(\frac{CRL}{r_0^2}\right)\de,$  we have $v' \in  \mS_K(p')$ and further, $\mS_K(p')$ contains a $r_0/2$-disc of dimension $D-d$  centered at $v'.$
\end{lemma}
\begin{proof}
We know that \beq\forall q \in \MM, \forall v_1 \in (T_p \MM)^\perp,\,\,\langle q - p, v_p + r_0 v_1\rangle \leq 0. \lab{eq:r0}\eeq
Fix $q \in \MM$ to be a point that maximizes $\langle v', q \rangle$ on $\MM.$ Then, $v'$ belongs to the closure of $\mS(q).$ Let $$v_1 :=  \frac{\Pi_{(T_p\MM)^\perp}(q - p)}{\|\Pi_{(T_p\MM)^\perp}(q - p)\|}.$$ Then by (\ref{eq:r0}) and by using Lemma~\ref{l:nu_p} in (\ref{eq:64.1}), 
\beq \langle q - p, - v_p\rangle & \geq & \langle q - p, r_0 v_1\rangle\\ & = & r_0  \|\Pi_{(T_p\MM)^\perp} (q - p)\|\\
& \geq & r_0 \left(\frac{1}{2R}\right)\|q - p\|^2.\lab{eq:64.1}\eeq

Since $p \in \MM,$ we also know that for any vector $v_q \in \mathrm{cl}\, \mS_K(q),$ $\langle v_q, p - q \rangle \leq 0.$
Take $v_q $ to be $v'$.
It follows that $$\langle q - p, v' - v_p\rangle \geq \left(\frac{r_0}{2R}\right) |q - p|^2.$$
Therefore, $$|q - p| \leq \left(\frac{2R}{r_0}\right)|v' - v_p|.$$
It remains to be shown that not only is $v'$ in $\mathrm{cl}\, \mS_K(q),$ but so is a $r_0/2$-disc of dimension $D-d$  centered at $v'.$ This follows from Theorem~\ref{thm:stable}.
\end{proof}

As a consequence of Lemma~\ref{lem:8.1}, we see that the preimage  $\pi^{-1}(\MM \cap B_\de^D(x))$ under the fiber map $\pi: B_1^D(0) \mapsto \MM$ contains a $\frac{\de}{L}$-neighborhood of $\frac{x + D_0(x)}{2}.$

\section{Algorithm}\lab{sec:correct}

\begin{figure}
    \centering
        \includegraphics[width=2.5in]{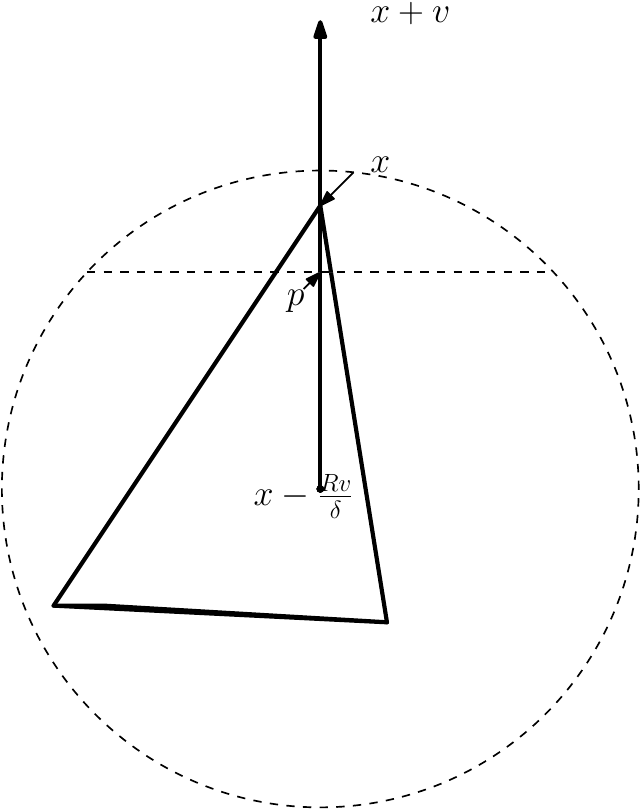} 
        \caption{The $D$-dimensional ball with center $x - \frac{R v}{\de}$ and radius  $\frac{R}{\de} + \eps$ contains the convex set $K$ (depicted above as a triangle).} \label{fig:ball1}
    \end{figure}

\subsection{Final Optimization}

We recall the definition of a Weak Optimization Problem (Definition~\ref{def:wopt}):\\
Given a vector $v \in \Q^n,$ and a rational number $\eps > 0$, either
\ben
\item[O1.] find a vector $y \in \Q^n$ such that $y \in S(K, \eps)$ and   $v^Tx \leq v^Ty  + \eps$ for all $x \in S(K, -\eps),$ or
\item[O2.] assert that $S(K, -\eps)$ is empty.
\een
We recall also the definition of a Weak Validity Problem (Definition~\ref{def:wval}):\\
Given a vector $c \in \Q^n,$ a rational number $\gamma$, and a rational number $\eps > 0$, either
\ben
\item assert that $c^Tx \leq \gamma + \eps$ for all $x \in S(K, -\eps),$ or
\item assert that $c^T x \geq \gamma - \eps$ for some $x \in S(K, \eps)$\\(i.\,e., $c^Tx \leq \gamma$ is almost invalid).
\een
Given an oracle that solves a Weak Validity Problem, \cite{Lov} shows how the ellipsoid algorithm can be used to solve  a Weak Optimization Problem with a polynomial run-time. {Note that this algorithm calls the algorithm Find-distance polynomially many times. We assume that it uses in each step, new, independently 
chosen sample points.}
We denote by $\pi^{alg}_\eps(v)$, the (random) output of the algorithm that takes as input, the vector $v$ and a threshold $\eps$,  and outputs a corresponding solution of the Weak Optimization Problem.  Note that this output is random because the oracle it queries,  uses random samples from $\mu \ast G_\sigma^{(D)}.$

We will now show that if a vector $v \in B_1^D(0) \cap \pi^{-1}(x)$ is the center of a $(D-d)$-dimensional $\de$-ball contained in $B_1^D(0)\cap\pi^{-1}(x),$ then the ellipsoid algorithm for solving the weak optimization problem with objective $v$ outputs a point $y$ that is close to $x$ in the Euclidean norm. 
\begin{lemma}
 If a unit vector $v_p \in  \pi^{-1}(p)$ is the center of a $(D-d)$-dimensional $\de$-ball contained in $\pi^{-1}(p),$ then, the $D$-dimensional ball with center $p - \frac{R v_p}{\de}$ and radius  $\frac{R}{\de}$ contains $K$, \ie
 $$B^D_{{R}/{\de}}\left(p - \frac{R v_p}{\de}\right) \supseteq K.$$
\end{lemma}
\begin{proof}

We know that \beq\forall q \in \MM\ \forall v_1 \in (T_p \MM)^\perp:\,\,\langle q - p, v_p + \de v_1\rangle \leq 0. \lab{eq:r0.1}\eeq
Fix $q \in \MM.$ Then, $v'$ belongs to the closure of $\mS(q).$ Let $$v_1 :=  \frac{\Pi_{(T_p\MM)^\perp}(q - p)}{\|\Pi_{(T_p\MM)^\perp}(q - p)\|}.$$ Then by (\ref{eq:r0.1}) and by using Lemma~\ref{l:nu_p} in (\ref{eq:64.1.1}), 
\beq \langle q - p, - v_p\rangle & \geq & \langle q - p, \de v_1\rangle\\ & = & \de  \|\Pi_{(T_p\MM)^\perp} (q - p)\|\\
& \geq & \de \left(\frac{1}{2R}\right)\|q - p\|^2.\lab{eq:64.1.1}\eeq
This gives us
\beq
  \|q- (p-(R/\de)v_p)\|^2 &=& \|q-p\|^2 +(R/\de)^2 - \langle q-p, -2(R/\de)v_p\rangle\\
 & \leq & \|q-p\|^2 +(R/\de)^2 - \|q - p\|^2\\
   &= & (R/\de)^2.
\eeq
Since this applies to an arbitrary point $q \in \MM,$ it proves the lemma.
\end{proof}

Let $v_i \in B^D_{1}(0)\setminus (B^D_{1-2\de}(0)).$
\begin{definition} Let $x_i = x_{v_i}$ be a maximizer of $\langle x, v_i\rangle$ over $\MM$, which is chosen arbitrarily in case there is more than one maximizer. \end{definition}
Let $p_i$ be the point $x_i - \frac{\eps v_i}{\|v_i\|}.$ 
Consider the halfspace $$H_{p_i} = \{{\mtext y\in \R^D}\,|\, \langle (y - p_i), v_i\rangle \geq 0\}.$$
 \begin{lemma} \lab{lem:9.2} 
A solution  $y_i$ of the Weak Optimization Problem which $c\eps$-approximately maximizes $\langle v_i,  z\rangle$ over $z \in K,$  satisfies $\|y_i - x_i\| < C\sqrt{\frac{\eps R}{\de_i}}.$
\end{lemma}
\begin{proof}
The diameter of $B^D_{\frac{R}{\de_i} + \eps}\left(x_i - \frac{R v_i}{\de_i \|v\|}\right) \cap H_{p_i}$ is bounded above by $C\sqrt{\frac{\eps R}{\de_i}}.$ Since $x_i$ belongs to this set,  the lemma follows.
By Lemma~\ref{lem:5.4},  the point $y_i$ which $c\eps$-approximately maximizes $\langle v_i,  z\rangle$ over $z \in K$ belongs to $$S(K, -\eps) \cap H_{p_i} \subseteq B^D_{\frac{R}{\de_i} + \eps}\left(x_i - \frac{R v_i}{\de_i \|v_i\|}\right) \cap H_{p_i}.$$ 
Therefore, the point $y_i$ satisfies $\|y - x_i\| < C\sqrt{\frac{\eps R}{\de_i}}.$
\end{proof}
\subsection{Algorithm for  identifying good choices of $v$.}

\begin{figure}
    \centering
        \includegraphics[width=3.5in]{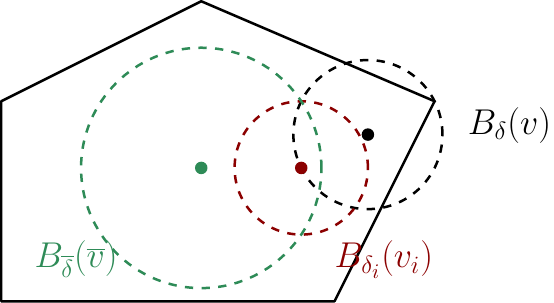} 
        \caption{There is at least one point $i \in [\NN]$,  such that $\de_i > cr_1\de.$ The convex set represents a two dimensional cross section of  $\mS_K(x_v).$  The balls depicted are the intersections of the affine span of this cross section with the the $D$-dimensional balls refered to in the picture above. } \label{fig:Bdelta}
    \end{figure}

We also need a procedure that handles the situation when $r_0$ is small, \ie when $v$ is close to the boundary of the outer normal cone it belongs to. This procedure must  exclude cases when the point output by the optimization routine applied to $v \in \R^D$ is far from the base point of the fiber containing $v$.  
Let $r_1$ be an a priori lower bound on the radius of the largest $(D-d)$-dimensional ball contained in $\mS_K(x),$ for any $x \in \MM.$ Note that by Lemma~\ref{l:thick outer cone},  $N_K(x)$ contains a $(D-d)$-dimensional ball of radius $\frac{\tau}{R},$ centered at a unit vector,  and so rescaling that vector multiplicatively by $\frac{1}{1 + ({\tau}/{R})},$  we see that we may take 
\beq r_1 := \frac{\tau}{\tau + R}.\lab{eq:95}\eeq
Let \beq r_0 := \frac{r_1}{2}.\lab{eq:r_0}\eeq We choose $\de $ given by 
\beqs \de = \min\left(\left(\frac{\eps}{2RL^2}\right)^{\frac{1}{3}}, \left(\frac{r_0^2}{RL}+ 1\right )r_0\right).\lab{eq:der0}\eeqs
As a consequence,  \beq\left(\frac{RL}{r_0^2}\right)\de = \min\left(\sqrt{\frac{\eps R}{r_1\de}}, \left(\frac{RL}{r_0^2}+ 1\right )r_0\right).\lab{eq:der0}\eeq
Let \beq \eps' := \frac{r_1}{4L}, \eeq where $L$ is the Lipschitz constant appearing in Theorem~\ref{thm:stable}.  Let $ \{v_i|i \in [\NN]\}$ be an $\eps' \de$-net of the sphere $\partial B_\de^D((1 - \de)v)$.  (Note that $x_i = x_{v_i}$ is in general {\it not} equal to $x_v$.)

For each $i \in [\NN]$,  let $\de_i$ be the largest nonnegative real number $\de'$ such that $B^D_{\de'}(v_i)\cap (T_{x_i}\MM)^\perp \subseteq \mS_K(x_i).$

\begin{lemma}\lab{lem:9.3}
\ben \item
There is at least one index $i \in [\NN]$,  such that $\de_i > cr_1\de.$ 
\item If $v$ is at a distance greater than $r_0$ from $\partial N_k(x_v),$  for every index $i \in [\NN]$,  we have $\de_i > cr_1\de.$
\een
\end{lemma}
\begin{proof}
Let $\overline{v}$ denote a point in $\mS_K(x)$ that is the center of a $(D-d)$-dimensional ball contained in $\mS_K(x)$ of radius $r_1.$ Let  $$v':= \partial B_\de^D((1 - \de)v) \cap \{\la \overline{v} + (1 - \la)(1 - \de) v | \la \in [0, 1]\}.  $$ Thus,  $v'$ is the intersection of the line segment joining $\overline{v}$ and $(1-\de)v,$ with $\partial B_\de^D((1 - \de)v).$ Since $(1-\de)v$ and $\overline{v}$,  both belong to $\mS(x),$ it follows from the convexity of $\mS(x),$ and the fact that $\|(1 - \de)v - \overline{v}\| < 2$ that $$ B_{{r_1\de}/{2}}^D(v') \cap (T_x\MM)^\perp \subseteq \mS(x).$$
Let $v_j$ be a point in   $\{v_i|i \in [\NN]\}$ at a minimal distance from $v'$.  By Theorem~\ref{thm:stable} we know that the the variation of $\mS_K(x)$ as a function of $x,$ measured in Hausdorff distance,  is $L$-Lipschitz.  Since  $\eps' = \frac{r_1}{4L},$ it follows that $\de_j >  cr_1\de.$ This proves part 1.  of  the lemma.  To see part 2.  of the lemma,  note that  if $v$ is at a distance greater than $r_0$ from $\partial N_k(x_v),$ by (\ref{eq:der0}),  $\de < r_0,$  $$ \partial B_{\de(1 - \de)}^D((1 - \de)v) \cap (T_x\MM)^\perp \subseteq \mS(x),$$ and so we are done by Theorem~\ref{thm:stable}.
\end{proof}

Recall that  $r_0 = \frac{\tau}{2\tau + 2R}$ by (\ref{eq:r_0}).
\begin{lemma} Let $v$ be sampled from the uniform probability measure on $\partial B_1(0).$ Then,
\beqs \p\left[\dist(v, \partial N_K(x_v)) \geq r_0\right] \geq 2^{-2D}\left(\frac{\tau}{R}\right)^{D+d}.\eeqs
\end{lemma}
\begin{proof}
Let $$G := \left\{v \in B_1(0) \big | \, \dist\left(\frac{v}{\|v\|}, \partial N_K(x_v)\right) \geq r_0\right\}.$$
By Lemma~\ref{lem:7.8}, 
$$
 \frac{\vol_D(G)}{\omega_D} 
 \ge \left( \frac{r_0\tau}{R} \right)^d
 \inf_{p\in\MM} \frac{\vol_{D-d}(\mS_K(p)\cap G)}{\vol_{D-d}(\mS_K(p))} .
$$ 
By (\ref{eq:95}) and (\ref{eq:r_0}), we see that for all $p \in \MM$, $\mS_K(p)\cap G$ contains a $(D-d)$-dimensional ball of radius $r_0$, 
and therefore, 
$$\left( \frac{r_0\tau}{R} \right)^d
 \inf_{p\in\MM} \frac{\vol_{D-d}(\mS_K(p)\cap G)}{\vol_{D-d}(\mS_K(p))} \geq r_0^D \left(\frac{\tau}{R}\right)^d.$$
Substituting the value of $r_0 = \frac{\tau}{2\tau + 2R} \geq \frac{\tau}{4R},$ the Lemma follows.
\end{proof}

Let us assume that $\eps < \frac{c \tau}{d}.$  Let \beq \NN_0 = \left\lceil\left(\frac{4R}{\tau}\right)^{D+d} \left(\frac{V}{\omega_d \eps^d}\right)\right\rceil. \eeq
Let $\eta \in (0, 1)$ be a parameter that bounds from above the probability with which the algorithm is permitted to fail.

\subsection{Algorithm Find-points}
\begin{algorithm}[H]

{{\bf Algorithm Find-points}

Input  parameters: $\eps>0$, $D,\NN,p, \NN_0,\in \mathbb Z_+$ 
and the sample points $X_1,X_2,\dots, X_{N_3}\in \R^D$, where $N_3= p\NN_0\NN$.}

\ben 
\item For $1 \leq j \leq \NN_0,$ do the following. \ben
\item Let $v^{(j)}$ be sampled from the  uniform measure on $\partial B_1(0).$
\item Let $ \{v^{(j)}_i|i \in [\NN]\}$ be an $\eps' \de$-net of the sphere $\partial B_\de^D((1 - \de)v^{(j)})$. 

\item {Use the algorithm  Weak-Optimization-Oracle to find the solution
$y_i^{(j)}$ for the Weak Optimization Problem which $c\eps$-approximately maximizes $\langle v_i^{(j)},  z\rangle$ over $z \in K,$ for each $i \in [\NN].$ {{This step uses $p$ samples.}}  If any of the algorithms  Weak-Optimization-Oracle
fail, output that the algorithm has failed, otherwise proceed to the next step.}

\item If  we have $\|y_1^{(j)} - y_i^{(j)}\| < C\sqrt{\frac{\eps R}{r_1\de}}, $ for each $i \in [\NN]$,   then output $y^{(j)} := y_1^{(j)}$,  otherwise output no point and declare that $v^{(j)}$ is within $r_0$ of $\partial N_K(\pi(v^{(j)})).$
\een
\een
\end{algorithm}
\begin{lemma}
{The number $N_3$ of the needed sample points $X_i$, $i=1,2,\dots,N_3$ as well
as} the number of the arithmetic operations performed during the execution {of the algorithm} Find-points is bounded above by \beq \NN\NN_0 \exp\left(\tilde{\Omega}\left(\left(\frac{RL\sigma}{\eps\tau^2}\log \frac{V}{\tau^d}\right)^2\right)\right) \log \left( \eta^{-1}\right).\eeq  where $L \leq CR(\Lambda + \tau^{-2})$ is the constant from the statement of Theorem~\ref{thm:stable}. 
\end{lemma}
\begin{proof}
{Recall (see the algorithm   Find-Distance,  formula  (\ref{eq: N co}))} that in order to get a uniform additive estimate on  $ \Gamma(H_{\gamma', b}, \frac{\eps}{RL})/|\MM|$ to within an accuracy of  
$$\mathbf{a} :=\frac{c\eps}{RL} \frac 1 {V}  \Gamma_{\eps/({RL})},$$
 with probability greater than $1 - {\eta_1}$,
it suffices to take $N$ samples, where 
\beq N\mathbf{a}^2  & \geq & C\log\left(  {\eta_1}^{-1} (N/D)^D\right)\\
& = &  C \log  {\eta_1}^{-1} + CD \log \frac{N}{D}. \eeq 
This is satisfied for 
\beq
N = \tilde{\Omega}\left(\mathbf{a}^{-2} D + \mathbf{a}^{-2} \log  {\eta_1}^{-1}\right),
\eeq
which can be bounded above by \beqs \exp\left(\tilde{\Omega}\left(\left(\frac{RL\sigma}{\eps\tau}\log \frac{V}{\tau^d}\right)^2\right)\right) \log \left(  {\eta_1}^{-1}\right) . \eeqs

{When we use $\eta_1=\frac{\eta}{\NN \NN_0}$, 
the each use of the algorithm   Find-Distance, called in the algorithn Weak-Optimization-Oracle,
outputs a correct solution at least with probability
 $1 - \frac{\eta}{\NN \NN_0}$. As the  algorithm   Find-Distance is called at most ${\NN \NN_0}$
 times, this yields that algorithm Find-points  outputs correct solutions at least with probability $1-\eta$.
 This proves the claim.}
\end{proof}

For constant values for  the  parameters $\sigma, R, \Lambda, \tau, V, \eta, \eps$ and $d <c\sqrt{\log\log n},$ this bound is dominated by the contribution of $\NN_0$, which is bounded by $\exp\left(CD\log\frac{R}{\tau}\right) < n^{C},$ where we recall that $D < \left(\frac{CV}{\tau^d}\right)^{\frac{d}{2} + 1}.$
For the purposes of analysis of the algorithm, we isolate the following subroutine, consisting of parts (c) and (d) of find-points,  which we term the $\de$-ball tester.
\subsection{Algorithm $\de$-ball tester}
\begin{algorithm}[H]
{\bf Algorithm $\de$-ball tester}

Input  parameters: $\eps>0$, $D, p, N\in \mathbb Z_+$
and the sample points $X_1,X_2,\dots, X_{\NN_4}\in \R^D$, where $\NN_4= \NN p $.
\ben 

\item {For each $i \in [\NN]$, use the algorithm Find-points to find the   solutions $y_i$  of the Weak Optimization Problem which $c\eps$-approximately maximizes $\langle v_i,  z\rangle$ over $z \in K$. {{This step uses $p$ samples.}}
If any of the algorithms  Find-points fails, output that the algorithm $\de$-ball tester has failed,
otherwise proceed to the next step.} 

\item If  we have $\|y_1 - y_i\| < C\sqrt{\frac{\eps R}{r_1\de}}, $ for each $i \in [\NN]$,   then output $y_1$,  otherwise output no point and declare that $v$ is within $r_0$ of $\partial N_K(x_v).$
\een
\end{algorithm}
\begin{theorem}
The 
$\de$-ball tester takes as input $v\in \partial B_1(0)$ and returns an output that satisfies the following properties.
\ben \item If $v$ is at a distance greater than $r_0$ from $\partial N_K(x_v),$ it returns a point $y_1$ that is $C\sqrt{\frac{\eps R}{r_1\de}}$-close to $x_v.$
\item If $v$ is at a distance less or equal to $r_0$ from  $\partial N_K(x_v),$ it either returns a point $y_1$ that is $C\sqrt{\frac{\eps R}{r_1\de}}$-close to $x_v$,  or outputs no point,  together with a declaration that $v$ is within $r_0$ of $\partial N_K(x_v).$
\een
\end{theorem}
\begin{proof}
The $\de$-ball tester returns the point $y_1$ corresponding to $v_1$ if and only if $\|y_1 - y_i\| < C\sqrt{\frac{\eps R}{r_1\de}}, $ for each $i \in [\NN]$,  otherwise it returns no point and a declaration of error.  If $v$ is at a distance greater than $r_0$ from $\partial N_K(x_v),$ by Lemma~\ref{lem:8.1},  for every $v'' \in B^D_{\de}((1 - \de)v),$ 
\ben
\item[(a)] $\|x_{v''} - x_v\| <  \left(\frac{CRL}{r_0^2}\right)\de.$
\item[(b)] $v''$ is the center of a $(D-d)$-dimensional ball of radius $\frac{r_0}{2}$ contained in $N_K(x_{v''}).$
\een
These facts, together Lemma~\ref{lem:9.2}, imply that for every $v_i, i \in [\NN]$ we have that the solution  $y_i$ of the Weak Optimization Problem which $c\eps$-approximately maximizes $\langle v_i,  z\rangle$ over $z \in K,$  satisfies $\|y_i - x_i\| < C\sqrt{\frac{\eps R}{\de_i}},$ where by Lemma~\ref{lem:9.3},  $\de_i > cr_1r_0> cr_1\de$ This proves part 1. of this theorem.
To see part 2., we apply part 1. of  Lemma~\ref{lem:9.3}, together with  Lemma~\ref{lem:9.2}, and observe in view of (\ref{eq:der0}),  that there exists a $y_j$ that is $C\sqrt{\frac{\eps R}{r_1\de}}$-close to $x_v.$ A necessary condition for the $\de$-ball tester returning a point,  is that this point $y_j$ is $C\sqrt{\frac{\eps R}{r_1\de}}$-close to $y_1.$ Therefore,  if a point is returned,    that point must be $C\sqrt{\frac{\eps R}{r_1\de}}$-close to $x_v.$ 
\end{proof}

\begin{corollary}
With probability greater or equal to $ 1 - C\eta$, the output of find-points is a $C\sqrt{\frac{\eps R}{r_1\de}}$ net of $\MM_0$, which we call $\mathcal{X}_1.$ 
\end{corollary}

\section{Implicit description of output manifold $\MM_{\mtext rec}$}\lab{sec:12}
Let $\eps'' := C\sqrt{\frac{\eps R}{r_1\de}}.$ Plugging the above corollary into subsection 5.5 of \cite{filn0}, and using the results in Sections 6 and 7 of that paper, we obtain an algorithm that computes an implicit description of a manifold $\MM_{\mtext rec}$ of reach greater than $c\tau/d^6$ such that $d_H(\MM_{\mtext rec}, \MM) < Cd\eps''.$ We give a brief overview of how this is done, below.
A subnet $\mathcal{X}_2$ at a scale $\frac{c\tau}{d}$ of the net $\mathcal{X}_1$ is fed into a subroutine $FindDisc$ below and as an output we obtain a family of discs of dimension $d$ that cover $\MM_0$ in the sense that the $n$-dimensional balls with the same centers and radii, cover $\MM.$ These discs are fine-tuned as mentioned below.
\subsection{$FindDisc$}

\begin{algorithm}[H]
{\bf Algorithm FindDisc}

{Input: Set $X_0$.}
\ben 
\item Let $x_1$ be a point that minimizes $ |1 - |x- x'||$ over all $x' \in X_0$.
\item Given $x_1, \dots x_m$ for $m \leq d-1$, choose $x_{m+1}$ such that $$\max(|1-|x- x'||, |\langle x_1/|x_1|, x'\rangle|, \dots, |\langle x_m/|x_m|, x'\rangle|)$$ is minimized among all $x' \in X_0$ for $x'= x_{m+1}$.
\een

Output points $x_1, \dots x_d$.
\end{algorithm}

We start with a putative of tangent disc on a point of the net  using $FindDisc.$ This is fine tuned to get essentially optimal tangent disc as follows. We view the tangent disc locally as the graph of a linear function for data and obtain a nearly optimal linear function using convex optimization, with quadratic loss as the objective, and the domain being a convex set parameterizing the flats near the putative flat.

\subsection{Bump functions}
For each center $p_i$ of each output disc of radius $r$, 
 consider the bump function $\tilde{\a}_i(x)$ given by 
 $$\tilde{\a}_i(p_i + rv) = c_i(1 - \|v\|^2)^{d+2},$$
 for any $v \in B_n$ and $0$ otherwise. Let 
 $$\tilde{\a}(x):= \sum_i \tilde{\a}_i(x).$$
 Let 
 $$\a_i(x) = \frac{\tilde{\a}_i(x)}{\tilde{\a}(x)}$$
for each $i.$ 




\subsection{Weights}
It is possible to choose $c_i$ such that for any $z$ in a $\frac{r}{4d}$ neighborhood is $\MM$, 
$$c^{-1} > \tilde{\a}(z) > c,$$
where $c$ is a small universal constant. Further, such $c_i$ can be computed using no more than $|\mathcal{X}_1|(Cd)^{2d}$ operations involving vectors of dimension $D.$


\subsection{An approximate gradient of the squared distance function}

We consider the \underline{scaled setting} where $r=1$. Thus, in the new Euclidean metric, $\tau \geq Cd^C$.
Let $\Pi^i$ be the orthogonal projection of $\R^n$ onto the $(n-d)-$dimensional subspace containing the origin that is orthogonal to the affine span of $D_i$.
Recall that the $p_i$ are the centers of the discs $D_i$ as $i$ ranges over $[N_3]$.
We define the function $F_i:U_i \ra \R^n$ by $F_i(x) =\Pi^i (x - p_i)$. Let $\cup_i U_i = U$.
We define $F:U \ra \R^n$ by \beq \label{eq: F function} F(x) = \sum_{i \in [N_3]} \a_i(x) F_i(x).\eeq

\begin{figure}\centering
\includegraphics[width=5.0cm]{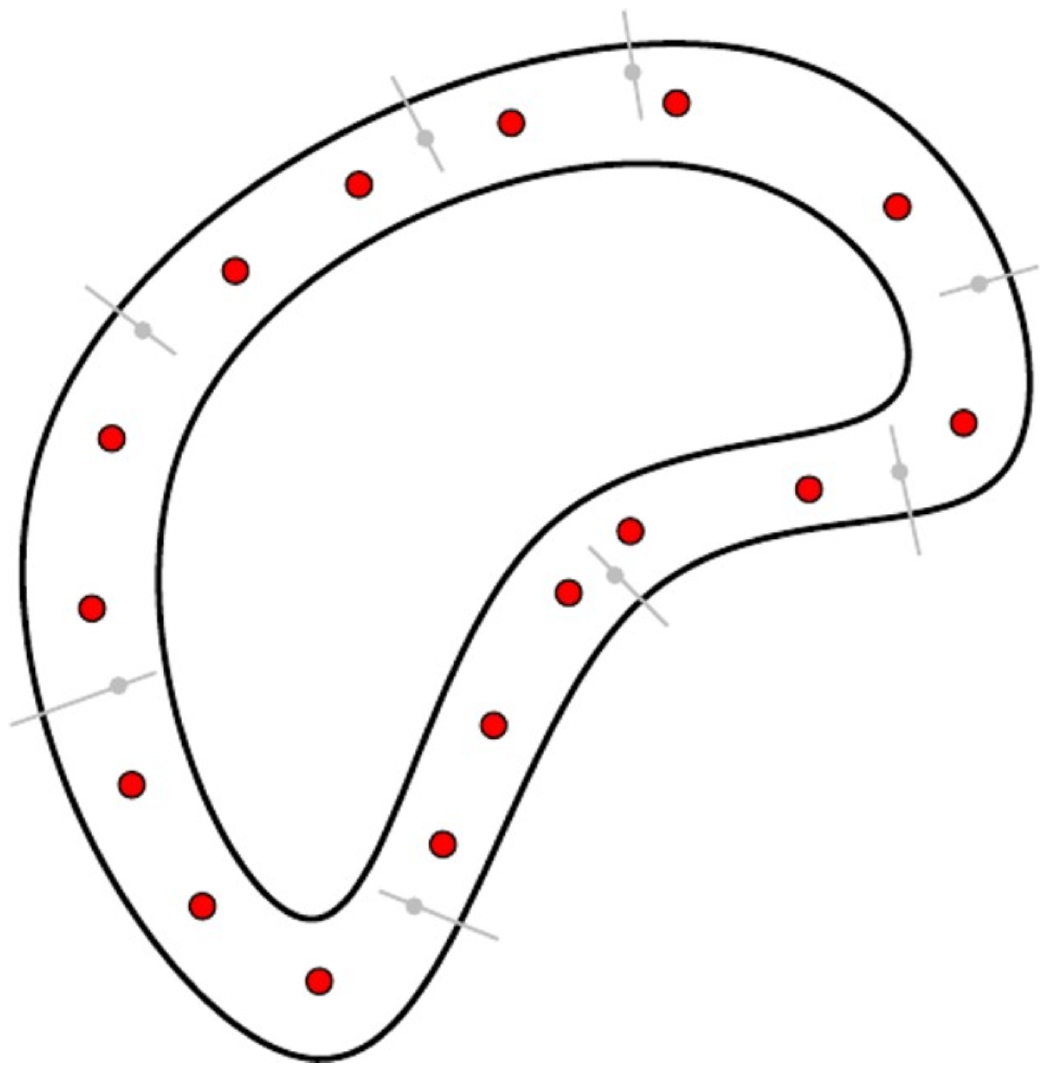}

\caption{Using data from the net, we construct an output manifold. The grey lines represent fibers of a certain approximate extension of the normal bundle of $\MM$, that is used to construct the output manifold. The manifold itself is defined by an equation of the form $\Pi_x F(x) = 0$ where $\Pi_x$ is an orthogonal projection on to the fiber at $x$, and $F$ is a certain vector valued function.}
\end{figure}

\subsection{An approximate (extension of) the normal bundle to a base that is a tubular neighborhood of $\MM$}

 Given a symmetric matrix $A$ such that $A$ has $n-d$ eigenvalues in $(1/2 , 3/2)$ and $d$ eigenvalues in $(-1/2, 1/2)$, let $\Pi_{hi}(A)$ denote the projection in $\R^n$ onto the span of the eigenvectors of $A$, corresponding to the largest $n-d$ eigenvalues.

\begin{definition} \label{def:6.9} For $x \in \cup_i U_i$, we define $\Pi_x = \Pi_{hi}(A_x)$ where $A_x = \sum_i \a_i(x) \Pi^i$. 
\end{definition}

\subsection{The output manifold}

Let $\widetilde{U}_i$ be defined as the $\frac{cr}{d}-$Eucidean neighborhood of $D_i$ intersected with $U_i$. Given a matrix $X$, its Frobenius norm $\|X\|_F$ is defined as the square root of the sum of the squares of all the entries of $X$. This norm is unchanged when $X$ is premultiplied  or postmultiplied by orthogonal matrices (of the appropriate order). Note that $\Pi_x$ is $\C^2$ when restricted to $\bigcup_i \widetilde U_i$, because the $\a_i(x)$ are $\C^2$ and when $x$ is in this set, $c < \sum_i \widetilde\alpha_i(x) < c^{-1}$, and for any $i,j$ such that $\a_i(x) \neq 0 \neq \a_j(x)$, we have $\|\Pi^i - \Pi^j\|_F < Cd\de$.
\begin{definition}\label{def:7}
Let $F_{\mtext rec}(x)=\Pi_x (F(x))$, where $F$ and $\Pi_x$ are constructed in formula
\eqref{eq: F function}  and Definition \ref{def:6.9}.
The output manifold $\MM_{\mtext rec}$ is the set of all points $x \in \bigcup_i \widetilde{U}_i$  such that 
$F_{\mtext rec}(x) = 0$, that is,
\beq
\MM_{\mtext rec}=\{x\in \R^n\ |\ F_{\mtext rec}(x) = 0\}.\lab{eq:master}\eeq
\end{definition}

\subsection{Analysis of reach}
Thus  $\MM_{\mtext rec}$ is the set of points $x \in \bigcup_i \widetilde{U_i}$ such that 
\beqs \Pi_{hi}(\sum_{i\in [N_3]} \a_i(x)\Pi^i)(\sum_{i \in [N_3]} \a_i(x)\Pi^i(x - p_i)) = 0 \eeqs
and
\beqs  \Pi_{hi}(\sum_i \a_i(x)\Pi^i) = \frac{1}{2\pi i} \left[\oint_\gamma (zI - (\sum_i \a_i(x)\Pi^i))^{-1}dz\right]\eeqs
using diagonalization and Cauchy's integral formula, and so

\beqs\label{MMo} \frac{1}{2\pi i} \left[\oint_\gamma (zI - (\sum_i \a_i(x)\Pi^i))^{-1}dz\right] \left(\sum_i \a_i(x)\Pi^i(x - p_i)\right) = 0\eeqs
where $\gamma$ is the circle of  radius $1/2$ centered at $ 1$.
The proof of the bound on the reach (Theorem 7.4 in \cite{filn0}) hinges on the above representation.

\subsection{Final result}

The additional computational cost in going from the net to $\MM_{\mtext rec}$ is exponential in $D$, nearly linear in $n$, and polynomial in $1/\eps''$. 
{These and the above considerations prove Theorem \ref{thm:main}.}

\section{Open questions and remarks} 
\ben \item Can any of the following constraints be relaxed while allowing the  parameters $\sigma, R, \Lambda, \tau, V, \eta, \eps$ to be arbitrary constants independent of $n$, and preserving the polynomial time guarantee?
\bit 
\item [i.] $\MM \subseteq \partial K$ ?
\item [ii.] $d < c \sqrt{\log\log n}.$
\eit
\item 
Can the guarantee on the reach of the output manifold be improved?
We currently have  $\mathrm{reach}(\MM_{\mtext rec})  \geq \frac{C\tau}{d^6}.$
\een




\section{Acknowledgements}
We are grateful to Somnath Chakraborty for several helpful discussions.

Ch.F. was partly supported by the US-Israel Binational Science Foundation grant number 2014055, AFOSR grant FA9550- 12-1-0425 and NSF grant DMS-1265524. S.I. was partly supported RFBR grant 20-01-00070, M.L. was supported by PDE-Inverse project of the European Research Council of the European Union and 
Academy of Finland, grants 273979 and 284715, and H.N. was partly supported by a Swarna Jayanti fellowship and the Infosys-Chandrasekharan virtual center for Random Geometry.
{\mtext Views and opinions expressed are 
those of the authors only and do not necessarily reflect those of the European
Union or the other funding organizations.  Neither the European Union nor the
 other funding organizations can be held responsible for them.}


\bibliographystyle{alpha}

\appendix

\section{Some auxiliary lemmas from \cite{filn0}}
\begin{lemma}[Lemma A.1 \cite{filn0}] \label{lem:g1}
Suppose that $\MM$ is a compact $d$-dimensional embedded $\C^2$-submanifold of $\R^m$ whose volume is at most $V$ and reach is at least $\tau.$ Let \beqs U:= \{y\in \R^m\big||y-\Pi_xy| \leq \tau/4\} \cap  \{ y\in \R^m\big||x-\Pi_xy| \leq \tau/4\}.\eeqs 
Then, $$\Pi_x(U \cap \MM) = \Pi_x(U).$$
\end{lemma}

\begin{lemma}[Lemma A.2 \cite{filn0}]\label{lem:6}
Suppose that $\MM \in \G(d, m, V, \tau)$. Let $x \in \mathcal \MM$ and \beqs \widehat U:= \{y\in \R^m\big||y-\Pi_xy| \leq \tau/8\} \cap  \{y \in \R^m\big||x-\Pi_xy| \leq \tau/8\}.\eeqs 
There exists a $C^{2}$ function $F_{x, \widehat U}$ from $\Pi_x( \widehat U)$ to $\Pi_x^{-1}(\Pi_x(0))$ such that
\beqs  \{ y + F_{x, \widehat U}(y) \big | y \in \Pi_x(\widehat U)\} = \MM \cap \widehat U.\eeqs 
Secondly,  for $ \de \leq \tau/8$,
 let $z \in \MM \cap \widehat U$ satisfy $|\Pi_x(z) - x| = \de.$ Let $z$ be taken to be the origin and let the span of the first $d$ canonical basis vectors be denoted $\R^d$ and let $\R^d$ be a translate of $Tan(x)$. Let the span of the last $m-d$ canonical basis vectors be denoted $\R^{m-d}$. In this coordinate frame, let a point $z' \in\R^m$ be represented as  $(z'_1, z'_2)$, where $z'_1 \in \R^d$ and $z'_2 \in \R^{m-d}$. 
 By Lemma~\ref{cl:g1sept}, there exists an $(m-d) \times d$ matrix $A_z$ such that  \beq \label{eq:36.1} Tan(z) = \{(z'_1, z'_2)| A_z z'_1 - I z'_2 = 0\}\eeq where the identity matrix is $(m-d) \times (m-d)$. Let $z\in \MM \cap \{z\big||z-\Pi_xz| \leq \de\} \cap \{z\big||x-\Pi_xz| \leq \de\}$. Then $\|A_z\|_2 \leq 15 \de/\tau.$ Lastly, the following upper bound  on the second derivative of $F_{x, \widehat U}$ holds for  $y \in \Pi_x( \widehat U)$. $$\forall {v \in \R^d} \ \forall {w \in \R^{m-d}}:\ \ \langle \partial_v^2 F_{x, \widehat U}(y), w\rangle \leq 
\frac{C|v|^2|w|}{\tau}.$$
\end{lemma}
\end{document}

%% file: submanifold_macros.tex
\newtheorem{theorem}{Theorem}[section]
\newtheorem{proposition}{Proposition}[section]
\newtheorem{lemma}[theorem]{Lemma}
\newtheorem{corollary}[theorem]{Corollary}

\newtheorem{definition}{Definition}[section]

\newtheorem{claim}{Claim}[section]

\newcommand{\E}{\ensuremath{\mathbb E}}
\newcommand{\R}{\ensuremath{\mathbb R}}

\newcommand{\F}{\ensuremath{\mathcal F}}
\newcommand{\FF}{\ensuremath{\mathcal F}}

\newcommand{\reach}{\mathrm{reach}}

\newcommand{\lab}{\label}  \newcommand{\ra}{\ensuremath{\rightarrow}}  \def\a{{\mathbf{\alpha}}} \def\de{{\mathbf{\delta}}}   
  \def\beq{\begin{eqnarray}} \def\eeq{\end{eqnarray}} \def\ben{\begin{enumerate}}
\def\een{\end{enumerate}}
 \def\bit{\begin{itemize}}
\def\eit{\end{itemize}}
 \def\beqs{\begin{eqnarray*}} \def\eeqs{\end{eqnarray*}} \def\bel{\begin{lemma}} \def\eel{\end{lemma}}
\newcommand{\N}{\mathbb{N}} \newcommand{\Z}{\mathbb{Z}} \newcommand{\Q}{\mathbb{Q}} \newcommand{\C}{\mathcal{C}} \newcommand{\CC}{\mathcal{C}}
    
   \newcommand{\p}{\mathbb{P}}
\newcommand{\PP}{\mathcal P}    
 \newcommand{\MM}{\mathcal M}\newcommand{\NN}{\mathcal N} \newcommand{\la}{\lambda}  
 \def\a{{\mathbf{\alpha}}}  \def\eps{{\epsilon}}  \def\ie{i.\,e.\,} 
\def\vol{\mathrm{vol}}

\renewcommand{\a}{\alpha}

\newcommand{\tmu}{\tilde{\mu}}
\newcommand{\tS}{\tilde{S}}

\newcommand{\RR}{\mathbb{R}}

\newcommand{\B}{\mathcal{B}}
\newcommand{\dist}{dist}

\newcommand{\G}{\mathcal{G}}

\renewcommand{\H}{\mathbb{H}}

\newcommand{\oc}{\overline{c}}

\newcommand{\beqn}{\begin{equation}}
\newcommand{\eeqn}{\end{equation}}